\documentclass[11pt, final]{article}

\usepackage[utf8]{inputenc} 
\usepackage{lmodern} 
\usepackage{a4wide, amsmath, amssymb, mathrsfs}        
\usepackage{amsthm}         
\usepackage[numbers, sort&compress]{natbib}         
\usepackage{enumitem} 
\setlist[enumerate]{label={\rm(\roman*)}} 
\usepackage[a4paper, margin=1in]{geometry} 
\usepackage{stmaryrd} 
\usepackage{authblk} 

\usepackage[hyphens]{url} 
\usepackage{hyperref} 
\hypersetup{colorlinks=true, linkcolor=teal, citecolor=violet} 

\usepackage{tikz-cd}
\usepackage{tikz}

\newtheorem*{theorem*}{Theorem}

\newtheorem{theorem}{Theorem}[section]

\theoremstyle{definition}
\newtheorem{definition}[theorem]{Definition}

\theoremstyle{definition}

\theoremstyle{plain}
\newtheorem{proposition}[theorem]{Proposition}

\theoremstyle{plain}
\newtheorem{lemma}[theorem]{Lemma}

\theoremstyle{plain}
\newtheorem{corollary}[theorem]{Corollary}

\theoremstyle{definition}
\newtheorem{remark}[theorem]{Remark}

\theoremstyle{definition}

\theoremstyle{plain}
\newtheorem{observation}[theorem]{Observation}

\theoremstyle{plain}

\theoremstyle{plain}

\numberwithin{equation}{section}

\def\integers{\mathbb{Z}}
\def\diff{\mathrm{d}}

\def\reals{\mathbb{R}}
\def\nat{\mathbb{N}}

\def\diam{\operatorname{diam}}
\def\dist{\operatorname{{dist}}}
\def\spt{\operatorname{spt}}

\def\conv{\operatorname{conv}}

\def\H{\mathcal{H}}
\def\Hn{\mathcal{H}^n}

\def\metric{d}

\newcommand{\dom}{\operatorname{Dom}}

\newcommand{\tLip}{\textnormal{Lip}}

\newcommand{\Id}{\operatorname{Id}}
\newcommand{\E}{\mathcal{E}}
\newcommand{\F}{\mathcal{F}}
\def\teps{\tilde{\varepsilon}}

\title{Perturbations of measures and sets having curves in $d$ directions}
\author{Jakub Tak\'a\v{c}}
\affil{University of Warwick\\
University of Trento\\}
\begin{document}
\maketitle

\begin{abstract}
    We show that whenever a separable subset $S$ of a complete metric space $X$ admits a $d$-dimensional weak tangent field, the set $S$ is close to being $d$-dimensional in the following sense. Whenever $\mu$ is a Borel finite measure on $X$ supported on $S$, then a typical $1$-Lispchitz map (in the sense of Baire category) into a Euclidean space maps $\mu$-almost all of $S$ into a set of Hausdorff dimension at most $d$. When taking $d=0$, this implies that any $1$-purely unrectifiable set is typically carried into a Hausdorff $0$-dimensional set up to a $\mu$-null set. We show that the result is sharp in Euclidean spaces and, more generally, in strictly convex Banach spaces of finite dimension.
\end{abstract}
\section{Introduction}

Recall that a set is called $n$-rectifiable, if it can be covered, up to an $\Hn$-null set, by a countable number of Lipschitz images of subsets of $\reals^n$. On the other hand a set $S$ is called $n$-purely unrectifiable, if for every $n$-rectifiable set $E$, one has $\Hn(S\cap E)=0$. Here and hereafter $\Hn$ is the $n$-dimensional Hausdorff measure. In this paper, we study the geometric idea that $n$-purely unrectifiable sets are, in some sense, close to being at most $(n-1)$-dimensional.

This is motivated by the previous work of Bate and Li \cite{BateLi} and Bate and Weigt \cite{BW}. It follows from their result that an $n$-purely unrectifiable set $S$ of \emph{finite} Hausdorff measure essentially admits a so-called $(n-1)$-dimensional weak tangent field. This can be used to show that an arbitrarily small $1$-Lipschitz perturbation of the identity may be found, say $f$, satisfying $\H^n(f(S))=0$ \cite[Theorem 1.1]{B}. In this sense, the set $S$ is close to being $\H^n$-null. However, since the argument crucially relies on the fact that $S$ has an $(n-1)$-dimensional weak tangent field, a natural question is whether one can produce a perturbation $f$ as above with the stronger property $\dim_H(f(S))\leq n-1$, where $\dim_H$ stands for the Hausdorff dimension.

The condition that $S$ be of finite Hausdorff measure is, in all of the aforementioned results, crucial. Therefore, it is also of interest to see what kind of a result, say in the spirit of the said perturbation result of Bate \cite{B}, can be obtained for general purely unrectifiable sets $S$. 

\subsection{A geometric property of planar purely unrectifiable sets}\label{SS:geometric-properties}

To further motivate our work, we begin with a simple result that can be obtained from the classical Besicovitch projection theorem.

\begin{proposition}\label{P:easy-squashing}
    If $S\subset \reals^2$ is a compact, $1$-purely unrectifiable set with $\H^1(S)<\infty$, then there exists a sequence of $1$-Lipschitz maps $f_i$ converging to the identity uniformly such that $f_i(S)$ is finite for each $i\in \nat$.
\end{proposition}

\begin{proof}
    Recall the assertion of the Besicovitch projection theorem which states that for almost every direction $\theta$, the orthogonal projection to the line through the origin in the direction of $\theta$, say $p_\theta\colon \reals^2 \to L_\theta$, satisfies
\begin{equation*}
    \H^1(p_\theta(S))=0.
\end{equation*}
Thus, it is without loss of generality to assume that the orthogonal projections to the two canonical vectors, say $p_1$ and $p_2$, carry $S$ into a set of zero length.

By the definition of $\H^1$, for any $\varepsilon>0$, each of the sets $p_i(S)$ can be covered by a countable number of intervals such that the total sum of their lengths does not exceed $\varepsilon$. Since $S$ is compact, this collection can also be taken to be finite. 
Say we denote the union of these intervals by $G_i$. We now let $f_i\colon \reals\to\reals$ be given by
\begin{equation*}
    f_i(t)=\int_0^t\chi_{\reals\setminus G_i}.
\end{equation*}
Note that $f_i$ carries $p_i(S)$ into a finite set.
It follows that the map
\begin{equation*}
    f\colon \reals^2 \to \reals^2;\quad f(x)=(f_1(p_1(x)), f_2(p_2(x))),
\end{equation*}
carries $S$ into a finite set, is $1$-Lipschitz and satisfies
    $\lVert f -\Id \rVert_{\infty}\leq \sqrt{2}\varepsilon$.
\end{proof}

\subsection{Perturbations in metric spaces}

The previous proposition verifies a particular special instance of the idea discussed above, namely that $1$-purely unrectifiable sets of finite length are close to being $0$-dimensional (even finite). However, same method cannot be applied to $n$-purely unrectifiable sets with $n>1$ and, more crucially, to sets having infinite Hausdorff measure. To obtain results in these more general settings, we rely on different techniques based on \emph{metric} notions. The perturbation result of Bate serves as an entry point to our approach and therefore we now continue by presenting the result explicitly.

Given a metric space $X$, the symbol $\tLip_1(X,\reals^m)$ stands for the space of $1$-Lipschitz and bounded maps from $X$ to $\reals^m$. We equip the space with the supremum norm, making it into a \emph{complete} metric space.

Recall that a subset $H$ of a metric space is called \emph{residual} if it contains an intersection of a countable number of dense open sets. The Baire category theorem asserts that in a complete metric space, residual sets are dense. Therefore, residual sets are a suitable notion of ``large'' sets. In particular, if we wish to assert that a Lipschitz map $f$ may be approximated with maps having a given property (cf.~Proposition \ref{P:easy-squashing}), this immediatelly follows if we are able to show that the set of maps possessing the given property is residual. We shall say that a \emph{typical} element of the space $H$ satisfies the given property, say $(P)$, if the set of elements satisfying $(P)$ is residual.

The following theorem of Bate \cite{B} originally required an additional assumption, but it can be removed by applying the recent result of Bate and Weigt \cite{BW} obtaining thereby the formulation we state below.

\begin{theorem}[\cite{B, BW}]\label{T:bate}
    Let $X$ be a complete metric space.
    \begin{enumerate}
        \item\label{Enum:B1} If $S\subset X$ is purely $n$-purely unrectifiable and $\Hn(S)<\infty$, 
        then a typical $f \in \tLip_1(X,\reals^m)$ satisfies $\Hn(f(S))=0$.
        \item If $E\subset X$ is $n$-rectifiable, $\Hn(E)>0$ and $m\geq n$, then the set of maps $f \in \tLip_1(X,\reals^m)$ satisfying $\Hn(f(E))>0$ is open and dense.
    \end{enumerate}
\end{theorem}

Thus in particular, amongst sets of finite $\H^n$ measure, the $n$-purely unrectifiable ones are characterised by the fact that they can be perturbed, with a Lipschitz map, into an $\Hn$-null set. In this sense, $n$-purely unrectifiable sets of finite measure are close to being $\Hn$-null.

Let us now examine closely the requirement of $\H^n(S)<\infty$ in the first item of the theorem. Clearly, it may weakened to only requiring that $\H^n_{|S}$ be $\sigma$-finite. On the other hand, it may not be removed entirely, as can be seen by the following example. Suppose $S\subset \reals^2$ is the Koch curve. Then any Lipschitz (even continuous) map $f\colon S \to \reals^2$ which is sufficiently close, in the supremum norm, to the identity, maps $S$ into a \emph{path-connected} set containing more than $1$ point. The image therefore must have positive length. Since $S$ is $1$-purely unrectifiable, it follows that Theorem \ref{T:bate} fails without the finite measure assumption.

Therefore the natural question is whether some analogous statement can be formulated, which would work without the finite measure assumption. The content of our main result is precisely such a statement, taking into consideration also the previously discussed idea of $n$-purely unrectifiable sets being ``nearly'' $(n-1)$-dimensional. We begin by stating the special case for $1$-purely unrectifiable sets.

\begin{theorem}\label{T:intro-main-1d}
    Let $X$ be a complete metric space and suppose $S\subset X$ is separable. If $S$ is purely $1$-unrectifiable, then for any finite Borel measure $\mu$ on $S$ and any $m\in\nat$, the set
    \begin{equation*}
        \{f\in\tLip_1(X, \reals^m): \exists N\subset S, \mu(N)=0\;\text{and}\; \dim_H f(S\setminus N)=0\}
    \end{equation*} 
    is residual in $\tLip_1(X, \reals^m)$.
\end{theorem}
If this result is applied to a natural measure on the Koch curve (e.g.~the one obtained from the self-similar construction, or the Hausdorff measure of the correct dimension) we can see that while almost all of the curve is typically mapped into a $0$-dimensional set, the full image is typically a non-trivial continuous curve. Thus, a typical $1$-Lipschitz map sends $\mu$-large sets into Hausdorff-small sets and $\mu$-small sets into Hausdorff-large sets.

We continue with the presentation of the higher-dimensional version of Theorem \ref{T:intro-main-1d}. Firstly, the notion of $n$-pure unrectifiability has to be replaced with a more rigid notion relying on weak tangent fields. 
Given a subset $S$ of $\reals^k$, we say that $S$ admits a $d$-dimensional weak tangent field if there is a map $\tau$ from $S$ to the Grassmanian of $d$-dimensional subspaces of $\reals^k$ with the following property. 
For every Lipschitz curve $\gamma$ in $\reals^k$ and for $\H^1$-a.e. $x\in \gamma\cap S$,
\begin{equation*}
    \gamma'(\gamma^{-1}(x))\in\tau(x).
\end{equation*}
For example, an embedded $d$-dimensional $C^1$-surface $M$ in $\reals^k$ admits a $d$-dimensional tangent field - this can be taken as the classical tangent field of the surface itself.
Also, it follows immediately from the definitions that $S$ is $1$-purely unrectifiable if and only if $S$ admits a $0$-dimensional weak tangent field. We remark that this notion can be made sense of in a general metric space by considering mappings into Euclidean spaces of higher dimension - see Definition \ref{def:atilde}.

For a set $S$ with $\H^n(S)<\infty$ it follows from the aforementioned results \cite{BateLi, BW} that $S$ is purely $n$-unrectifiable if and only if $S$ admits an $(n-1)$-dimensional weak tangent field up to an $\H^n$-null set. Thus, we argue that in the absence of the finite measure assumption, the correct condition to impose on $S$ is a dimensionality of a weak tangent field as opposed to pure unrectifiability. 

Our main result follows. Notice that by taking $d=0$ below, one recovers Theorem \ref{T:intro-main-1d}

\begin{theorem}\label{T:intro-squashing}
    Let $X$ be a complete metric space and $S\subset X$ a separable subset which admits a $d$-dimensional weak tangent field. Then for any finite Borel measure $\mu$ on $S$ and any $m\in\nat$, the set
    \begin{equation*}
        \{f\in\tLip_1(X,\reals^m): \exists N\subset S, \mu(N)=0\;\text{and}\; \dim_H(f(S\setminus N))\leq d\}
    \end{equation*}
    is residual in $\tLip_1(X,\reals^m)$.
\end{theorem}

Next, we are concerned with the sharpness of Theorem \ref{T:intro-squashing}. Firstly, we can now see that the Hausdorff dimension is not stable under perturations. Recall that given a finite Borel measure $\mu$ on some metric space $X$, its \emph{Hausdorff dimension}, $\dim_H \mu$, is the infimum of the set of those $s\in \reals$ for which there exists a $\mu$-null set $N$ with
\begin{equation*}
    \H^s(X\setminus N)<\infty.
\end{equation*}
Since the Hausdorff dimension is not stable under perturbations, we may ``forcefully'' stabilize it via the relaxation:
\begin{equation}\label{E:stable-id}
    \begin{split}
        \dim_H^{\textnormal{stable}}(\mu)&=\max_{m\in\nat}\inf\{s\in[0,\infty]: \exists f_i\in\tLip_1(\reals^k,\reals^k)\\
        & f_i\to \Id \;\text{uniformly on bounded sets and}
        \dim_H({f_i}_\#\mu)\leq s
        \},
    \end{split}
\end{equation}
for a finite Borel measure $\mu$ on the Euclidean space $\reals^k$. This may be modified for the metric setting via
\begin{equation}\label{E:stable-dense}
    \begin{split}
        \dim_H^{\textnormal{stable}}(\mu)&=\max_{m\in\nat}\inf\{s\in[0,\infty]: \{f\in\tLip_1(X,\reals^m): \dim_H({f}_\#\mu)\leq s\}\\
        & \text{is dense in $\tLip_1(X,\reals^m)$}
        \},
    \end{split}
\end{equation}
for a finite Borel measure $\mu$ on the complete metric space $X$.
The coincidence of the two definitions whenever $X$ is Euclidean is neither obvious, nor crucial for our argument. It does follow from the theory developed in the rest of this paper, see Corollaries \ref{C:final} and \ref{C:stability-for-Euclidean-Haus} and the discussion at the end of Section \ref{S:stability}.

Theorem \ref{T:intro-preserve} gives a one-sided estimate on the quantity $\dim_H^{\textnormal{stable}}(\mu)$ in terms of the weak tangent field. Namely, for a subset $S$ of a metric space $X$, we define its \emph{tangential dimension}, $\dim_T S$, by
\begin{equation*}
    \dim_T(S)=\min\{d\in\nat\cup\{0\}: \textnormal{$S$ admits a $d$-dimensional weak tangent field}\},
\end{equation*}
where we convene that the minimum of the empty set is $\infty$. For a finite Borel measure $\mu$ on the complete metric space $X$, we let its \emph{tangential dimension}, $\dim_T \mu$, be given by
\begin{equation*}
    \dim_T \mu = \min\{d\in \nat\cup\{0\}: \exists N\subset X, \mu(N)=0\;\textnormal{and}\; \dim_T(X\setminus N)\leq d\}.
\end{equation*}

Equipped with this notation, Theorem \ref{T:intro-squashing} in particular implies that 
\begin{equation*}
    \dim_H^{\textnormal{stable}}(\mu)\leq \dim_T \mu.
\end{equation*}

In the case when $X$ is a Euclidean space, we are able to also obtain the converse inequality. 
The principal intermediate result which we prove is that for measures in the Euclidean space (or, more generally, strictly convex finite dimensional Banach spaces) the tangential dimension is preserved under small perturbations.
\begin{theorem}\label{T:intro-preserve}
    Suppose $\mu$ is a finite Borel measure in the strictly convex finite dimensional Banach space $X$.
    Then there is a compact convex set $C\subset X$ and $\varepsilon>0$ such that whenever a $1$-Lipschitz map $f\colon X\to X$ satisfies $\lVert f - \Id\rVert_{\ell^\infty(C)}\leq \varepsilon$, then
    \begin{equation*}
        \dim_T(f_\#\mu)\geq \dim_T \mu.
    \end{equation*}
\end{theorem}

By relying on the theory of so-called Alberti representations, which we discuss extensively in Subsection \ref{sec:alberti}, it follows that the number of independent Alberti representations a measure has is preserved under perturbations. Thus, available results in literature \cite{DeRindler, BKO, ArRaDimension} almost immediately yield the partial converse to Theorem \ref{T:intro-squashing}.

\begin{corollary}\label{C:converse}
    Suppose $\mu$ is a Borel finite measure in $\reals^k$ with $\dim_T(\mu)\geq d\leq m$. Then the set
    \begin{equation*}
        \{f\in \tLip_1(\reals^k, \reals^m): \dim_H(f_\#\mu)\geq d\}
    \end{equation*}
    has non-empty interior.
\end{corollary}
The combination of Theorem \ref{T:intro-squashing} and the preceding Corollary \ref{C:converse} implies the equality 
\begin{equation*}
    \dim^{\textnormal{stable}}_H \mu=\dim_T \mu
\end{equation*}
for any Borel finite measure $\mu$ on a Euclidean space. A trivial, yet remarkable consequence of this equality is that $\dim_H^{\textnormal{stable}} (\mu)$ is an integer. In particular $\dim_H^{\textnormal{stable}} (\mu)$ is at most the integer part of $\dim_H (\mu)$.

\subsection{On some related results}
The aforementioned Alberti representations of a finite Borel measure $\mu$ on $\reals^k$ are used in the work of Alberti and Marchese \cite{AM} to construct the \emph{decomposability bundle} $V_\mu$ which is a $\mu$-measurable map on $\reals^k$  taking values in the Grassmanian of linear subspaces of $\reals^k$. This bundle can be seen as a kind of tangent field of $\mu$ and one of the main results of the paper \cite{AM} is a characterisation of $\mu$-a.e.~differentiability of Lipschitz functions $f\colon \reals^k \to \reals$ in terms of $V_\mu$. A seemingly different construction of a kind of tangent field of a measure is carried out by Bouchitt\'e et al \cite{BCJ}, which can be shown to be equivalent to the bundle of Alberti and Marchese \cite{AdolfoPrivate}.

Alberti and Marchese \cite{AM} need to construct Lipschitz mappings, which are not differentiable in a prescribed direction and on a prescribed set. At this point in time, the kind of construction they rely on is fairly well established in the literature. The main geometric point is that Lipschitz functions may be constructed that have ``bad behaviour'' in those direction in which there are no Lipschitz curves. Similar type of a construction is used, for example, in the aforementioned work of Bate \cite{B}, the paper of Schioppa \cite{SchDer} on derivations, the earlier work of Bate \cite{BateLDS} on Lipschitz differentiability spaces or the paper of Aliaga et al \cite{AGPP} on $1$-purely unrectifiable spaces. 

In order to prove our Theorem \ref{T:intro-squashing} we also rely on this type of construction, namely we explicitly make use of the form of this construction found in Bate's paper \cite[Proposition 3.5]{B}. The main point in our case is not a construction of a Lipschitz map which is not differentiable (in a prescribed direction) but rather a map that ``squashes'' nearby points together (also only in a prescribed direction).
The existence of these maps that squash points together is what allows us to construct (different) maps that effectively concentrate measure into sets that have small Hausdorff content of the given dimension. This is the core idea behind Theorem \ref{T:intro-squashing} and is carried out rigorously in Section \ref{S:Squashing}.

Finally, we discuss a connection of Marstrand's projection theorem \cite{Mars,Kaufman} to Theorem \ref{T:intro-squashing}. To make our point clear, let us first recall the Besicovitch--Federer projection theorem. The theorem states that for a purely $n$-unrectifiable subset $S$ of a Euclidean space, of finite measure, almost every orthogonal projection to an $n$-dimensional subspace carries $S$ into a null set. If $S$ is $n$-rectifiable and has positive measure, than the opposite conclusion holds. 
Therefore, if one simply takes the Besicovitch--Federer theorem and replaces ``almost every projection'' with ``a typical $1$-Lipschitz map'', the resulting statement is true; in fact it is a special case of Theorem \ref{T:bate} above due to Bate \cite{B}. Therefore the natural question is whether the same can be done for a variant of the Marstrand's projection theorem.

One formulation of said theorem found in the book of Mattila \cite[Corollary 9.8]{Mat} implies in particular that if a set $E\subset \reals^k$ has positive Hausdorff $s$-measure, then almost every projection of the set to an $m$-dimensional subspace with $m<s$ has Hausdorff dimension at least $m$.
If, on the other hand, $\dim E <m$, then the dimension of almost every projection to an $m$-dimensional subspace is at least $\dim E$ \cite[Corollary 9.4]{Mat}. In this sense almost all projections preserve the Hausdorff dimension. 
This is contrary to our results, which state that if typical $1$-Lipschitz maps are considered (as opposed to almost all projections) then preservation of Hausdorff dimension cannot be expected.

To illustrate our point on a specific example, if one considers once again the Koch curve in the plane equipped with the Hausdorff measure of the correct dimension, then almost every projection to a $1$-dimensional subspace has dimension one (in fact, the pushforward of the measure is absolutely continuous with respect to the Lebesgue measure and its density lies in $L^2$ \cite{Kaufman}), while a typical $1$-Lipschitz pushforward of the measure is Hausdorff $0$-dimensional.

\subsection{Structure of the paper}
Section \ref{S:prel} contains preliminaries needed in the rest of the paper. Sections \ref{S:Squashing} and \ref{S:residuality-Euclidean} are dedicated to proving Theorem \ref{T:intro-squashing} together with some additional results concerning certain cases when the target space $\reals^m$ is replaced by a finite dimensional $\ell^p$ space. We also provide the result in several forms, each essentially equivalent to every other. Section \ref{S:Squashing} concentrates on the density part of the argument (recall that in order to show that a set is residual, one needs to show density and openness of certain other sets). In this case, the openness part of the argument is much simpler and is contained in the proofs of the general statements in Section \ref{S:residuality-Euclidean}. Finally, Section \ref{S:stability} is concerned with proving Theorem \ref{T:intro-preserve} and Corollary \ref{C:converse}. 

\subsection{Acknowledgements}
    I would like to thank my PhD supervisor David Bate for many discussions concerning this work and a careful reading of the final version of this manuscript. I would also like to express my gratitude to Giovanni Alberti and Andrea Marchese. There is a somewhat complicated story behind the conception of the main ideas presented in this work. Firstly, Theorem \ref{T:intro-squashing} is a statement that Giovanni, David and Andrea worked on and tried to prove some time ago. They were able to come up with a rather complicated proof of the theorem in the special case where $\mu$ is some instance of a Hausdorff measure. It does not appear that their method can yield the result for a general Borel measure $\mu$. Therefore, David suggested this to me as an open problem. After some time, I came up with the proof of the general Theorem \ref{T:intro-squashing}.

    In order to show that the result is sharp (at least in Euclidean spaces), we first prove Theorem \ref{T:intro-preserve}. This is something that was advised to me by David, who has previously also discussed this result with Giovanni and Andrea. Nevertheless, the others suggested to me to publish the results as the only author.

    I would like to thank Adolfo Arroyo-Rabasa and Filip Rindler, the examiners of my PhD oral exam, for their helpful suggestions on the improvement of my thesis, which also included a previous version of this work. 

    Finally, I wish to express my gratitude to Milo\v{s} Ta\v{s}i\'c, whose support proved unexpectedly nourishing while I was working on this problem.

    \bigskip

    The author is supported by the Warwick Mathematics Institute Centre for Doctoral Training and by the European Union’s Horizon 2020 research and innovation programme (Grant agreement No.
    948021)
\section{Preliminaries}\label{S:prel}

In this section we collect some basic preliminary facts needed to carry out the arguments in the rest of the paper. We rely largely on existing literature, however, there are several basic results concerning \emph{measurability} of certain constructions that are not available and we must develop the theory ourselves . 

\subsection{Basic measure theory}
In the following we recall some basic measure theory in general metric spaces, which we will use without further reference.
We call a metric space $X$ which is complete and separable a \emph{Polish metric space}.
If $X$ is a set and we say that $\mu$ is a measure on $X$, we mean that $\mu$ is an outer measure. If $X$ is a set and $\Sigma$ is a $\sigma$-algebra on $X$ and we say that $\mu$ is a measure on $(X,\Sigma)$, we mean that $\mu$ is an outer measure on $X$ and each set in $\Sigma$ is $\mu$-meausrable. If the $\sigma$-algebra with which we are working is clear from the context, we shall omit specifying it.
We say that a measure $\mu$ on the metric space $X$ is Borel, if the $\sigma$-algebra of $\mu$-measurable sets contains the Borel $\sigma$-algebra. Recall that if $X$ is a Polish metric space and $\mu$ is a $\sigma$-finite Borel measure on $X$, then $\mu$ is \emph{inner regular}, i.e.~for every Borel set $B\subset X$ there exists a sequence of compact sets $K_i\subset B$ such that
\begin{equation*}
    \mu(B\setminus \bigcup_i K_i)=0.
\end{equation*}
In particular, if $\mu(B)<\infty$, then to each $\varepsilon>0$, there exists a compact set $K\subset B$ such that $\mu(B\setminus K)<\varepsilon$.

We say that a measure $\mu$ on a space $X$ is supported on a $\mu$-measurable set $E\subset X$ if $\mu(X\setminus E)=0$. We let $\spt \mu$ be the least closed set $F\subset X$ such that $\mu(X\setminus F)=0$.
If $X$ is a complete metric space and $\mu$ is a Borel $\sigma$-finite measure on $X$ which is supported on a separable set, then $\mu$ is inner regular. Indeed, if $\mu$ is supported on the separable set $S$, then $\spt\mu\subset \overline{S}$ and $\overline{S}$ is a Polish metric space.

\begin{definition}
    Let $X$ be a metric space and $s\in[0,\infty)$. For a subset $E\subset X$ and $\delta\in(0,\infty]$, we define the \emph{($s$-dimensional) Hausdorff $\delta$-content} of $E$ by
    \begin{equation*}
        \H^s_\delta(E)=\inf\{\sum_i\diam(E_i)^s: E\subset \bigcup_i E_i;\; E_i\subset X\}.
    \end{equation*}
    We call $\H^k_\infty(E)$ the \emph{($s$-dimensional) Hausdorff content} of $E$.
    The \emph{$s$-dimensional Hausdorff measure} of $E$ is defined as
    \begin{equation*}
        \H^s(E)=\sup_{\delta>0}\H^s_\delta(E).
    \end{equation*}
\end{definition}
We let $\dim_H(X)$ be the infimum of those values of $s$ for which $\H^s(X)=0$. If no such $s$ exists, we let $\dim_H(X)=\infty$. The quantity $\dim_H(X)$ is called the Hausdorff dimension of $X$.

Given a measure space $(R,\mu)$, and $f_i, f$ measurable functions valued in some metric space $(X, \metric)$, we say that $f_i\to f$ in measure, if for every $\varepsilon>0$,
\begin{equation*}
    \mu(\{x\in R: \metric(f_i(x), f(x))\geq \varepsilon\})\to 0.
\end{equation*}
We recall the basic fact that if $\mu$ is $\sigma$-finite, then this notion of convergence is induced by a metric.
The following averaging lemma is the main tool we use to show convergence in measure.

\begin{lemma}\label{L:av-new}
    Suppose $(R,\mu)$ is a probability space and $A\in[0,1]$. Whenever $\xi\colon R \to \reals$ a measurable function such that
    \begin{equation*}
        \xi\leq 1 \quad\text{a.e.}\quad\text{and}\quad \int_R \xi\diff\mu\geq 1-A,
    \end{equation*}
    then
    \begin{equation*}
        \mu\{(x\in R: \xi(x)>1-\sqrt{A})\}\geq 1-\sqrt{A}.
    \end{equation*}
\end{lemma}
\begin{proof}
    Let us denote $\lambda=\mu\{(x\in R: \xi(x)>1-\sqrt{A})\}$. Then we have
    \begin{equation*}
        1-A\leq \int_R \xi \diff \mu = \int_{\{(x\in R: \xi(x)>1-\sqrt{A})\}} \xi\diff\mu + \int_{\{(x\in R: \xi(x)\leq1-\sqrt{A})\}}\xi\diff\mu\leq \lambda + (1-\lambda)(1-\sqrt{A}).
    \end{equation*}
    Using elementary arithmetic yields
    \begin{equation*}
        \lambda\geq 1- \sqrt{A}.
    \end{equation*}
\end{proof}

\begin{lemma}\label{L:ac-with-error}
    Let $(R,\Sigma)$ be a measurable space and suppose $\mu$ and $\nu$ are finite measures on $R$ with $\mu\ll\nu$.
    To each $\alpha>0$, there exists $\varepsilon>0$ such that whenever $\tilde{\nu}$ is another finite measure on $R$ with $\tilde\nu\leq \nu$ and $(\nu-\tilde\nu)(R) \leq \varepsilon$, then there is a measurable set $E\subset R$ such that $\mu(R\setminus E)\leq \alpha$ and $\mu_{|E}\ll \tilde\nu$.
\end{lemma}
\begin{proof}
    Let $g=\frac{\diff \mu}{\diff \nu}\in L^1(\nu)$. For any $r>0$, there exists $t_0>0$ such that if $G\subset R$ is measurable and $\nu(G)\leq t_0$, then 
    \begin{equation*}
        \mu(G)=\int_G g \diff \nu \leq \alpha.
    \end{equation*}
    Let $\varepsilon>0$ be such that $\sqrt{\varepsilon\nu(R)}< t_0$. Let $f=\frac{\diff \tilde\nu}{\diff \nu}$. Then, by our assumption,
    \begin{equation*}
        f\leq 1\quad\text{and}\quad \frac{1}{\nu(R)}\int f\diff\nu\geq 1- \frac{\varepsilon}{\nu(R)}.
    \end{equation*}
    Thus, by Lemma \ref{L:av-new}, the set $E=\{x\in R: f(x)\geq 1- \sqrt{\frac{\varepsilon}{\nu(R)}}\}$ satisfies
    \begin{equation*}
        \frac{1}{\nu(R)}\nu(E)\geq 1- \sqrt{\frac{\varepsilon}{\nu(R)}},
    \end{equation*}
    i.e.
    \begin{equation*}
        \nu(R\setminus E)\leq \sqrt{\varepsilon\nu(R)}.
    \end{equation*}
    Whence, by our assumption on $\varepsilon$, $\mu(R\setminus E)\leq \alpha$. On the other hand, since $f$ is bounded away from zero on $E$, it is the case that $\nu_{|E}\ll \tilde\nu_{|E}$, which implies $\mu_{|E}\ll \tilde\nu$.
\end{proof}

\subsection{Hausdorff distance}
Let $(X, \metric_X)$ and $(Y, \metric_Y)$ be metric spaces. On the product $X\times Y$ we shall convene to define the maximum product metric:
\begin{equation*}
    \metric_{X\times Y}((x_0,y_0), (x_1,y_1))=\metric((x_0,y_0), (x_1,y_1))
    =\max\{\metric_X(x_0,x_1), \metric_Y(y_0,y_1)\}.
\end{equation*}
For a subset $E\subset X$ and $r\geq 0$, we let
\begin{equation*}
    B_X(E, r)=B(E,r)=\{x\in X: \dist(x,y)\leq r\}.
\end{equation*}
For a pair of non-empty subsets $E,F\subset X$, we define their \emph{Hausdorff distance} by
\begin{equation*}
    \metric_{H(X)}(E,F)=\max\{\inf\{r\geq 0: E\subset B(F,r)\},\inf\{r\geq 0: F\subset B(E,r)\}\}.
\end{equation*}
By $H(X)$ we denote the collection of non-empty closed \emph{bounded} subsets of $X$. With this notation, $(H(X), \metric_{H(X)})$ is a metric space. If $X$ is complete, separable or compact, respectively, so is $H(X)$ \cite{AT}. The following lemma follows immediately from the definition of the maximum product metric.

\begin{lemma}\label{L:product-hausdorff-est}
    Suppose $X,Y$ are metric spaces and let $p_X\colon X\times Y \to X$ and $p_Y\colon X\times Y \to Y$ be the coordinate projections. Let $E,F\subset X\times Y$ be non-empty and suppose $r\geq 0$ is such that
    \begin{equation*}
    \begin{split}
        &B_X(p_X(E), r)\supset p_X(F);\\
        &B_Y(p_Y(E), r)\supset p_Y(F).
    \end{split}
    \end{equation*}
    Then $B_{X\times Y}(E,r)\supset F$.
\end{lemma}

\begin{lemma}\label{L:isometry-product-hausdorff}
    Suppose $X,Y$ are metric spaces.
	The identity map \[\Id\colon H(X)\times H(Y) \to H(X\times Y)\] given by
    \begin{equation*}
        \Id(E,F)=E\times F\quad\text{for $(E,F)\in H(X)\times H(Y)$}
    \end{equation*}
    is an isometry.
\end{lemma}
\begin{proof}
    Let $E,G\in H(X)$ and $F,A\in H(Y)$. From Lemma \ref{L:product-hausdorff-est}, it follows that
    \begin{equation*}
        \metric_{H(X\times Y)}(E\times F, G\times A)\leq \max\{\metric_{H(X)}(E,G), \metric_{H(Y)}(F,A)\}.
    \end{equation*}
    The converse inequality is an immediate consequence of the fact that the projection maps $p_X$ and $p_Y$ are $1$-Lipschitz, i.e.
    \begin{equation*}
        d_X(p_X(x_0,y_0), p_X(x_1, y_1))\leq d_{X\times Y} ((x_0, y_0),(x_1, y_1)),
    \end{equation*}
    and similarly for $p_Y$.
\end{proof}

\subsection{Spaces of Lipschitz maps and residuality}
Suppose $(X, \metric_X)$ and $(Y, \metric_Y)$ are metric spaces and $f\colon X\to Y$ is a map. We call the least number $L$ such that
\begin{equation*}
    \metric_Y(f(x_0),f(x_1))\leq L \metric_X(x_0,x_1)\quad\text{for all $x_0,x_1\in X$,}
\end{equation*}
the \emph{Lipschitz constant} of $f$ and denote it by $\tLip(f)$. If $\tLip(f)$ is finite, we call $f$ Lipschitz and for $L\in[0,\infty)$, we say that $f$ is $L$-Lipschitz if $\tLip(f)\leq L$.

Recall the McShane extension theorem which asserts that whenever $X$ is a metric space, $E\subset X$ and $f\colon E\to \reals$ is $L$-Lipschitz, then there exists an extension of $f$ to $X$ which is also $L$-Lipschitz.

Suppose now $X$ is a metric space and $(Y, \lVert \cdot \rVert_Y)$ is a normed linear space. For $L\in[0,\infty)$ we define the metric space
\begin{equation*}
    \tLip_L(X,Y)=\{f\colon X\to Y: f\;\text{is $L$-Lipschitz and bounded}\}
\end{equation*}
equipped with the supremum metric
\begin{equation*}
    d(f,g)=\lVert f-g \rVert_\infty=\sup_{x\in X} \lVert g(x)-f(x)\rVert_Y.
\end{equation*}
If $Y$ is complete, then so is $\tLip_L(X,Y)$.

Given a metric space $X$ a set $H\subset X$ is called \emph{residual}, if it contains a countable intersection of dense and open sets. Recall the Baire category theorem, which asserts that if $X$ is complete and $H\subset X$ is residual, then $H$ is dense. In particular, if $Y$ is complete, then residual subsets of $\tLip_L(X,Y)$ are dense.

\subsection{Curve fragments and Alberti representations}
\label{sec:alberti}
Let us fix a metric space $X$.
We let
\begin{equation*}
    \Gamma(X)=\{\gamma:\dom\gamma\to X: \gamma\in\tLip(\dom\gamma, X);\; \dom\gamma\subset[0,1]\}.
\end{equation*}
The elements of the set $\Gamma(X)$ are called \emph{Lipschitz curve fragments}. Note that formally, by the set-theoretic definition of a function, each $\gamma\in \Gamma(X)$ is a closed subset of $X\times [0,1]$. Thus, we equip $\Gamma(X)$ with the \emph{Hausdorff metric} inherited from $H(X\times [0,1])$. We denote the metric by $\metric_{\Gamma(X)}$.
We define the space
\begin{equation*}
    P(X)=\{(\gamma, t)\in \Gamma(X)\times [0,1]: t\in\dom\gamma\}.
\end{equation*}
and equip it with the product metric inherited from $\Gamma(X)\times [0,1]$.

The notion of a curve fragment will later allow us to introduce Alberti representations and weak tangent fields, which are geometric notions which form the basis of the present paper. However, more set-theoretical aspects of the theory need to be developed first. This is made imperative by our efforts is Section \ref{S:stability}, where the geometric idea of what we do is relatively clear (and essentially contained in Subsection \ref{SS:perturbations}), however, the rigorous proofs require us to show measurability of certain operations. This is by no means obvious and requires the theory which we introduce in the remainder of this subsection and the subsequent subsection.

The main point being that given a measure $\eta$ on the space $\gamma$, we shall need to do a kind of restriction of $\eta$-a.e.~$\gamma$ to some subset of $\dom(\gamma)$. The problem is that the particular subset to which we restrict could, in principle, vary wildly from one curve fragment to another. The two main tools we use are the following lemma and the ``natural'' embedding of $\Gamma(X)$ into $P(\Gamma(X))$ introduced below.

\begin{lemma}[Lemma 3.6 \cite{BW}]\label{L:restrictions}
	Suppose $X$ is a Polish metric space and $K\subset X$ is compact. Let \[\Gamma_K(X)=\{\gamma\in\Gamma(X): \gamma^{-1}(K)\not=\emptyset\}.\] Then $\Gamma_K(X)$ is closed and the map $R_K\colon \Gamma_K(X)\to\Gamma(X)$ given by 
    \begin{equation*}
        R_K(\gamma)=\gamma_{|\gamma^{-1}(K)}
    \end{equation*}
    is a Borel map. In particular, given any compact set $C\supset K$, the map $R_K\colon \Gamma_K(X)\to\Gamma(C)$ is Borel.
\end{lemma}

\begin{lemma}\label{L:H1-Borel}
    Let $X$ be a complete and separable metric space. The function $\H^1\colon \Gamma(X)\to [0,\infty)$ given by $\H^1(\gamma)=\H^1(\gamma(\dom\gamma))$ is Borel. In fact, it is a pointwise limit of upper semi-continuous maps.
\end{lemma}
\begin{proof}
    Firstly, recall that for any metric space $M$, any set $A\subset M$, and any $\delta>0$ one may alternatively define the Hausdorff content via
    \begin{equation*}
        \H^1_\delta(A)=\inf\{\sum_i\diam A_i: A\subset \bigcup_i A_i, \; \text{$A_i$ are open}\}.
    \end{equation*}
    Thus, the map
    \begin{equation*}
        \H^1_\delta\colon H(M)\to [0,\infty]
    \end{equation*}
    is upper semi-continuous at each compact $A\in H(M)$. Indeed, let $A_i$ be an admissible cover of $A$. Then, by compactness, there is a finite subcover $A_{i_1}, \dots, A_{i_n}$. It follows that there is some $r>0$ such that $B(A,r)\subset \bigcup_i A_i$. Thus, if $B\subset M$ has Hausdorff distance from $A$ at most $r$, it follows that $A_i$ also forms an admissable cover of $B$. In particular, if $B_k$ is a sequence of elements of $H(M)$ converging to $A$, it holds that
    \begin{equation*}
        \H^1_\delta(A)\geq \limsup_{k\to \infty} \H^1_\delta(B_k).
    \end{equation*}

    Suppose now $\gamma_k, \gamma \in \Gamma(X)$ and $\gamma_k\to \gamma$ in $\Gamma(X)$. Then by definition also $\gamma_k(\dom\gamma_k)\to \gamma(\dom\gamma)$ in $H(X)$. Thus, using the above paragraph for $M=X$, it follows that
    \begin{equation*}
        \H^1_\delta(\gamma(\dom(\gamma)))\geq \limsup_{k\to \infty}\H^1_\delta(\gamma_k(\dom\gamma_k)).        
    \end{equation*}
    That is, $\gamma\mapsto \H^1_\delta(\gamma(\dom(\gamma)))$ is upper semi-continuous on $\Gamma(X)$. It follows that $\H^1$, defined as above, is a pointwise limit of upper semi-continuous maps and therefore it is Borel.
\end{proof}

Combination of the two previous lemmas implies that for any compact set $K\subset X$, the functional
\begin{equation*}
    \gamma\mapsto \H^1(\gamma\cap K)
\end{equation*}
is Borel. 

 \begin{definition}
     Let $X$ be a complete and separable metric space and $\eta$ a finite Borel measure on $\Gamma(X)$. We define the measure $B(\eta)$ on $X$ given by
     \begin{equation*}
         B(\eta)(K)=\int_{\Gamma_K(X)}\int_{R_K(\gamma)}\diff\H^1_{R_K(\gamma)}\diff\eta(\gamma)\quad\text{for $K\subset X$ compact,}
     \end{equation*}
     and
     \begin{equation*}
         B(\eta)(G)=\sup\{B(\eta)(K): K\subset G; \;K\;\text{compact} \} \quad \text{for $G\subset X$ open,}
     \end{equation*}
     and
     \begin{equation*}
        B(\eta)(E)=\inf\{B(\eta)(G): G\supset E; \;G\;\text{open} )\}\quad \text{for $E\subset X$}.
     \end{equation*}
     This defines a Borel measure on $X$.
 \end{definition}
 One can verify that for any Borel set $E\subset X$, the map
 \begin{equation*}
     \gamma \mapsto \int_{\gamma^{-1}(E)}|\gamma'|\diff \H^1
 \end{equation*}
 is $\eta$-measurable and the formula
 \begin{equation}\label{E:formula-Barycenter}
     B(\eta)(E)=\int_{\Gamma(X)}\int_{\gamma^{-1}(E)}|\gamma'|\diff \H^1 \diff\eta(\gamma), \quad \text{for any Borel set $E\subset X$,}
 \end{equation}
 holds. The formula \eqref{E:formula-Barycenter} is the only formula for $B(\eta)$ we shall be using. 

Given a closed $E\subset X$, we identify $P(E)$ and $\Gamma(E)$ with the \emph{closed} subsets of $P(X)$ and $\Gamma(X)$ via
\begin{equation*}
    P(E)=\{(\gamma,t)\in P(X): \gamma\subset E \};\quad \Gamma(E)=\{\gamma\in\Gamma(X):\gamma\subset E\}.
\end{equation*}
Note that the topologies of $P(E)$ and $\Gamma(E)$ obtained by the construction coincide with those inherited from $P(X)$ and $\Gamma(X)$ after the identification above.

\begin{definition}\label{D:AR}
    Let $X$ be a complete and separable metric space and $\mu$ a finite measure on $X$. We say that a finite Borel measure $\eta$ on $\Gamma(X)$ is an \emph{Alberti representation of $\mu$}, if $\eta$ is supported on injective curve fragments and
    \begin{equation*}
        \mu\ll B(\eta).
    \end{equation*}
\end{definition}

\begin{lemma}\label{L:isom-embedding-curves}
    Let $X$ be a metric space and consider the map $E\colon \Gamma(X)\to \Gamma(P(X))$ given by
    \begin{equation*}
    \begin{split}
        &\dom(E(\gamma))=\dom\gamma \quad \text{for $\gamma\in\Gamma(X)$},\\
        &E(\gamma)(t)=(\gamma,t)\in P(X)\quad \text{for $\gamma\in \Gamma(X)$ and $t\in \dom (E(\gamma))$.}
    \end{split}
    \end{equation*}
    The map $E$ is an isometry.
\end{lemma}
\begin{proof}
    Using Lemma \ref{L:isometry-product-hausdorff} and the definition of $E$, for any $\gamma_0, \gamma_1\in \Gamma(X)$, we have
    \begin{equation*}
        \begin{split}
        &\metric_{\Gamma(P(X))}(E(\gamma_0), E(\gamma_1))\\
        &=\metric_{H(P(X)\times[0,1])}(\{\gamma_0\}\times \dom(\gamma_0)\times\dom(\gamma_0),\{\gamma_1\}\times \dom(\gamma_1)\times\dom(\gamma_1))\\
        &=\max\{\metric_{H(P(X))}(\{\gamma_0\}\times \dom(\gamma_0),\{\gamma_1\}\times\dom{\gamma_1}), \metric_{H([0,1])}(\dom(\gamma_0), \dom(\gamma_1))\}
        \end{split}
    \end{equation*}
    However, since $P(X)$ is a subset of $\Gamma(X)\times[0,1]$, it follows from Lemma \ref{L:isometry-product-hausdorff}, that the last expression above is equal to
\begin{equation*}
 \begin{split}
    &\max\{\metric_{H(\Gamma(X))}(\{\gamma_0\}, \{\gamma_1\}), \metric_{H([0,1])}(\dom(\gamma_0), \dom(\gamma_1)), \metric_{H([0,1])}(\dom(\gamma_0), \dom(\gamma_1))\}\\
    &=\max\{\metric_{\Gamma(X)}(\gamma_0, \gamma_1), \metric_{H([0,1])}(\dom(\gamma_0), \dom(\gamma_1))\}
    =\metric_{\Gamma(X)}(\gamma_0, \gamma_1),
 \end{split}
\end{equation*}
where the first inequality uses the fact that the Hausdorff distance of singletons agrees with the distance of their elements and the second equality uses the definition of distance in $\Gamma(X)$.
\end{proof}

Suppose $X$ is a metric space and $(Y,\metric)$ is another metric space, $E\subset X$ is closed and $\mathcal{E}_i, \mathcal{E}\colon P(X)\to Y$ are Borel measurable maps. We say that the sequence $\mathcal{E}_i$ $m$-converges to $\mathcal{E}$ on $E$, and write $\E_i\overset{m}{\to} \E$ on $E$, if for every $\gamma\in\Gamma(E)$, the functions $\E_i(\gamma,\cdot)$ converge in measure to $\E(\gamma, \cdot)$ on $\dom(\gamma)$.

\begin{lemma}\label{L:m-to-measure}
    Suppose $E\subset X$ is closed and $\E_i, \E\colon P(X)\to Y$ are Borel measurable maps with $\E_i\overset{m}{\to}\E$ on $E$. Let $\eta$ be a finite Borel measure on $\Gamma(X)$ which is supported on $\Gamma(E)$. Then $\E_i\to\E$ in measure on $P(X)$ with respect to the measure $(\eta\times\H^1_{[0,1]})_{|P(X)}$.
\end{lemma}
\begin{proof}
    Fix $\delta>0$. Consider for each $i\in\nat$ the function $F^i_\delta\colon \Gamma(X)\to \reals$ given by
    \begin{equation*}
        F^i_\delta(\gamma)=\H^1(\{t\in\dom\gamma: |\E_i(\gamma, t)-\E(\gamma, t)|\geq \delta\}).
    \end{equation*}
    Since each $F^i_\delta$ is the measure of a slice of the Borel measurable set $(|\E_i-\E|)^{-1}((-\infty, -\delta]\cup[\delta, \infty))$, Fubini's theorem implies that each $F^i_\delta$ is $\eta$-measurable and, moreover,
    \begin{equation}\label{E:curve-fubini}
        (\eta\times\H^1_{[0,1]})(\{(\gamma, t): |\E_i(\gamma, t)-\E(\gamma, t)|\geq \delta\}) = \int_{\Gamma(X)}F^i_\delta(\gamma)\diff \eta(\gamma).
    \end{equation}
    Since $\E_i \overset{m}{\to} \E$ on $E$, $F^i_\delta \to 0$ pointwise on $\Gamma(E)$. By Egorov's theorem, since $\eta$ is supported on $\Gamma(E)$, to each $\varepsilon>0$, we find some Borel set $\mathcal{U}\subset \Gamma(E)$ such that
    \begin{equation}\label{E:curve-small-set}
        \eta(\Gamma(X)\setminus \mathcal{U})<\varepsilon
    \end{equation}
    and $F^i_\delta \to 0$ uniformly on $\mathcal{U}$. Thus, for every $r>0$, we find some $i_0\in \nat$ such that for all $i\geq i_0$,
    \begin{equation}\label{E:curve-uniform}
        F^i_\delta(\gamma)\leq r \quad \text{for all $\gamma\in \mathcal{U}$}.
    \end{equation}
    Using \eqref{E:curve-fubini}, we have
    \begin{equation*}
    \begin{split}
        (\eta\times\H^1_{[0,1]})(\{(\gamma, t): &|\E_i(\gamma, t)-\E(\gamma, t)|\geq \delta\})
        =\int_{\mathcal{U}}F^i_\delta \diff \eta +\int_{\Gamma(X)\setminus\mathcal{U}}F^i_\delta \diff \eta\\
        &\overset{\eqref{E:curve-uniform}, \eqref{E:curve-small-set}}{\leq} r \eta(\mathcal{U})+ \eta(\Gamma(X)\setminus \mathcal{U})\leq r \eta(\Gamma(X))+\varepsilon.
    \end{split}
    \end{equation*}
    Since $\delta$ was arbitrary and $\eta(\Gamma(X))$ is finite, the statement follows by sending $\varepsilon, r \to 0$.
\end{proof}

\subsection{Restriction operators}\label{sec:restriction}
\begin{definition}
    Let $X$ be a metric space and $\mathcal{U}\subset \Gamma(X)$. A map $\mathfrak{R}\colon \mathcal{U} \to \Gamma(X)$ is called a \emph{restriction on $\mathcal{U}$}, if for every $\gamma\in\mathcal{U}$, $\dom(\mathfrak{R}(\gamma))\subset \dom(\gamma)$ and for each $t\in \dom (\mathfrak{R}(\gamma))$ it holds that $\mathfrak{R}(\gamma)(t)=\gamma(t)$. The restriction $\mathfrak{R}$ is called Borel if both $\mathcal{U}$ and $\mathfrak{R}\colon \mathcal{U} \to \Gamma(X)$ are Borel.
    The \emph{relative density} of the restriction $\mathfrak{R}$ is the map $\mathcal{D}\colon \mathcal{U}\to [0,1]$ given by
    \begin{equation*}
        \mathcal{D}(\gamma)=\begin{cases}\displaystyle
				\frac{\H^1(\mathfrak{R(\gamma)})}{\H^1(\gamma)}
					& \text{if $\H^1(\gamma)>0$},
					\\[\bigskipamount]
				1
					& \text{otherwise},
			\end{cases}
    \end{equation*}
\end{definition}
\begin{theorem}\label{T:density-est-for-restrictions}
    Suppose $X$ is a complete and separable metric space and $\mathfrak{R}\colon \mathcal{U} \to \Gamma(X)$ is a Borel restriction. Then the relative density $\mathcal{D}$ of $\mathfrak{R}$ is Borel. Suppose $\eta$ is a finite Borel measure on $\Gamma(X)$, supported inside $\mathcal{U}$. Then
    \begin{equation*}
    \begin{split}
        &B(\mathfrak{R}_\#\eta)\leq B(\eta);\\
        &(B(\eta)-B(\mathfrak{R}_\#\eta))(X)= \int_{\mathcal{U}} (1-\mathcal{D}(\gamma))\H^1(\gamma) \diff\eta(\gamma).
    \end{split}
    \end{equation*}
\end{theorem}
\begin{proof}
    The fact that the relative density is Borel follows from the fact that it is a composition of $\mathfrak{R}$ and $\H^1$, which are both Borel (by Lemma \ref{L:H1-Borel}).
    Since curve fragments $\gamma$ with $\H^1(\gamma)=0$ contribute nothing to the barycenter, it is without loss of generality to assume that $\mathcal{U}$ contains only curve fragments with positive length. Then $B(\mathfrak{R}_\#\eta)\leq B(\eta)$ follows immediately from the definition of restriction. Moreover,
    \begin{equation*}
        B(\mathfrak{R}_\#(\eta))(X)=\int_{\Gamma(X)} \int_{\gamma}\diff\H^1\diff \mathfrak{R}_\#(\eta)(\gamma)=\int_{\mathcal{U}}\int_{\mathfrak{R}(\gamma)}\diff\H^1\diff\eta(\gamma),
    \end{equation*}
    and, since $\eta$ is suported inside $\mathcal{U}$,
    \begin{equation*}
        B(\eta)(X)
        =\int_{\mathcal{U}}\int_{\gamma}\diff\H^1\diff\eta(\gamma).
    \end{equation*}
    Substracting these quantities yields
    \begin{equation*}
        (B(\eta)-B(\mathfrak{R}_\#(\eta)))(X)=\int_{\mathcal{U}}(\H^1(\gamma)-\H^1(\mathfrak{R}(\gamma))) \diff\eta(\gamma)
        =\int_{\mathcal{U}}(1-\mathcal{D}(\gamma))\H^1(\gamma) \diff \eta(\gamma).
    \end{equation*}
\end{proof}

\begin{corollary}
    Suppose $X$ is a complete and separable metric space and $\mathfrak{R}_i\colon \mathcal{U}\to \Gamma(X)$ is a sequence of Borel restrictions whose relative densities $\mathcal{D}_i$ converge to $1$ pointwise in $\mathcal{U}$. Suppose $\eta$ is a finite Borel measure in $\Gamma(X)$ supported on $\mathcal{U}$. Suppose further that every $\gamma\in\mathcal{U}$ satisfies $\H^1(\gamma)>0$ and that there is some $L<\infty$ such that every $\gamma\in \mathcal{U}$ is $L$-Lipschitz. Then $(B((\mathfrak{R}_i)_\#\eta)- B(\eta))(X)\to 0$. 
\end{corollary}
\begin{proof}
    Since the domain of each $\gamma \in \mathcal{U}$ is contained in $[0,1]$, it follows that $\H^1(\gamma)\leq L$ for every $\gamma\in \mathcal{U}$. Thus, the sequence $(1-\mathcal{D}_i)\H^1(\cdot)$, where each $\mathcal{D}_i$ is the relative density of $\mathfrak{R}_i$, is a uniformly bounded sequence of Borel maps converging pointwise to $0$ on $\mathcal{U}$. Since $\eta$ is a finite Borel measure, by dominated convergence theorem, we have
    \begin{equation*}
        \int_{\mathcal{U}} (1-\mathcal{D}_i(\gamma))\H^1(\gamma) \diff\eta(\gamma)\to 0,
    \end{equation*}
    and so the statement follows from Theorem \ref{T:density-est-for-restrictions}.
\end{proof}

Finally, we shall introduce a way to obtain restriction maps which are guaranteed to be Borel. This is obtained by considering restrictions which arise as \emph{slices} of compact sets in $P(X)$.

\begin{theorem}\label{T:Borel-restriction}
    Suppose $X$ is a complete separable metric space and $\mathcal{K}\subset P(X)$ is a compact set. Let $p\colon P(X)\to \Gamma(X)$ be the projection map $p(\gamma, t)=\gamma$.
    The set $\mathcal{U}=p(\mathcal{K})$ is a compact subset of $\Gamma(X)$.
    The map
    \begin{equation*}
        \mathfrak{R}\colon \mathcal{U}\to \Gamma(X)
    \end{equation*}
    given by
    \begin{equation*}
        \mathfrak{R}(\gamma)=\gamma_{|\{t\in\dom\gamma: (\gamma,t)\in\mathcal{K}\}}
    \end{equation*}
    is a well defined, Borel restriction on $\mathcal{U}$.
\end{theorem}
\begin{proof}
    The validity of the definition follows from the fact that since $\mathcal{K}$ is compact, also any set of the form
    \begin{equation*}
        \{t\in\dom\gamma: (\gamma,t)\in\mathcal{K}\}
    \end{equation*}
    is compact. Since the map $p$ is continuous, $\mathcal{U}$ is a continuous image of a compact set and thus compact. We have
    \begin{equation*}
        \mathfrak{R}= E^{-1}\circ R \circ E,
    \end{equation*}
    where $E\colon \Gamma(X)\to \Gamma(P(X))$ is defined in Lemma \ref{L:isom-embedding-curves} and $R$ is given by
    \begin{equation*}
        R(\Phi)=\Phi\cap \mathcal{K}=\Phi_{|\Phi^{-1}(\mathcal{K})}\quad \text{for $\Phi\in \Gamma(P(X))$ with $\Phi\cap\mathcal{K}\not=\emptyset$.}
    \end{equation*}
    By Lemma \ref{L:restrictions}, the map $R$ is Borel. Since $E$ is an isometry and $R$ maps $E(\Gamma(X))$ into $E(\Gamma(X))$, it follows that $\mathfrak{R}$ is Borel.
\end{proof}

\subsection{Cones and independence}
\begin{definition}
  \label{D:cones}
  Let $V$ be a finite dimensional Banach space. For $u\in V^*$
   and a \emph{width} $0<\theta<1$, we define the \emph{cone centred on $u$ of width $\theta$} to be
  \begin{equation*}
        C(u,\theta)= \{v\in V: \langle u, v \rangle \geq (1-\theta)\lVert v\rVert\}.  
  \end{equation*}
  We say that cones $C_1,\ldots, C_n \subset V$ are \emph{independent} if, for any choice of $v_i \in C_i \setminus \{0\}$ for each $1\leq i \leq n$, the $v_i$ are linearly independent. If $V=\reals^n$ we identify in the standard way $(\reals^n)^*$ and $\reals^n$ and for a vector $u\in\reals^n$ and a thickness $\theta$ we have
  \begin{equation*}
      C(u,\theta)=\{v\in V: u\cdot v \geq (1-\theta)\lVert v\rVert\}.
  \end{equation*}

  For $W\subset V$ a subspace, we define the ``conical complement'' of $W$ to be
  \begin{equation*}
      E(W,\theta)= \{v\in V : d(v,W) \geq (1-\theta)\|v\|\}.
  \end{equation*}
  Note that both of the above sets become wider as $\theta\to 1$.
\end{definition}

For the remainder of this subsection we consider a fixed metric space $X$ and a fixed finite dimensional Banach space $V$.

\begin{definition}\label{d:alberti-direction}
  Let $F\colon X \to V$ be Lipschitz and $D$ a set of the form $C(w,\theta)$ or $E(W,\theta)$.
  We say that a curve fragment $\gamma \in \Gamma$ is in the \emph{$F$-direction} of $D$ if
  \[(F\circ\gamma)'(t) \in D\setminus \{0\}\]
  for $\H^1$-a.e.\ $t\in \dom\gamma$.
\end{definition}

\begin{definition}
  \label{def:atilde}
  Fix a Lipschitz $F\colon X \to V$ and an integer $d \leq \dim V$.
  For $0<\theta<1$ we define $\tilde A(F, d, \theta)$ to be the set of all $S\subset X$
  for which there exists a Borel decomposition $S=S_{1} \cup \ldots \cup S_{M}$ and $d$-dimensional subspaces $W_{i} \subset V$
  such that, for each $1\leq i \leq M$, $\H^1(\gamma\cap S_{i})= 0$
  for every $\gamma\in \Gamma$ in the $F$-direction of $E(W_{i},\theta)$.
  Further, we define $\tilde A(F,d)$ to be the set of all $S\subset X$ that belong to $\tilde A(F,d,\theta)$ for each $0<\theta<1$.

  Let $S\subset X$ be Borel.
  A Borel $\tau \colon S \to G(m,d)$ is a \emph{$d$-dimensional weak tangent field to $S$ with respect to $F$} if, for every $\gamma\in\Gamma(X)$,
  \begin{equation*}
    (F\circ\gamma)'(t) \in \tau(\gamma(t)) \quad \text{for } \H^1 \text{-a.e. } t\in\gamma^{-1}(S).
  \end{equation*}
\end{definition}

\begin{definition}\label{D:AR-independence}
    Suppose $\eta$ is a finite Borel measure on $\Gamma(X)$.
    Suppose $F\colon X\to V$ is a Lipschitz map and $C$ a cone in $V$.
    We say that $\eta$ \emph{goes in the $F$-direction of $C$} if $\eta$-a.e.~$\gamma$ goes in the direction of $C$.
    If $X=\reals^m$ and $F$ is the identity, we say that $\eta$ \emph{goes in the direction of $C$}.
    We say that finite Borel measures $\eta_1, \dots, \eta_d$ on $\Gamma(X)$ are \emph{independent} if there exist independent cones $C_1, \dots, C_d$ in $\reals^m$ and a Lipschitz map $F\colon X \to \reals^m$ such that each $\eta_i$ goes in the $F$-direction of $C_i$.
\end{definition}

\begin{definition}
    We say that a set $S\subset X$ \emph{admits a $d$-dimensional weak tangent field} if for every $m\in\nat$ and every Lipschitz map $F\colon S \to \reals^m$, there exists a $d$-dimensional weak tangent field to $S$ with respect to $F$. We let
    \begin{equation*}
        \dim_T(S)=\min\{d\in\nat \cup\{0\}: S\; \textnormal{admits a $d$-dimensional weak tangent field}\}.
    \end{equation*}
    Given a Borel measure $\mu$ on $X$, we let
    \begin{equation*}
        \dim_T(\mu)=\min\{d\in\nat\cup\{0\}: \exists N\subset X,\; \mu(N)=0,\; \dim_T(S\setminus N)\leq d\}.
    \end{equation*}
    In both cases we convene that the minimum of an empty set is $\infty$.
\end{definition}

\begin{definition}
    Let $\mu$ be a $\sigma$-finite Borel measure on $X$. Let $d\in\nat$. We say that $\mu$ \emph{has a piece with $d$ independent Alberti representations} if there exists a compact set $K\subset X$ such that $\mu(K)>0$ and $\mu_K$ has $d$ independent Alberti representations. 
\end{definition}

The number of independent Alberti representations can be related to the notion of $\tilde{A}$ sets via \cite[Theorem 2.11]{B}, which we present explicitly in a slightly different, but equivalent, form.
\begin{theorem}\label{T:dimA-iff-AFd}
    Suppose $\mu$ is a $\sigma$-finite measure on $X$ and let $d\in\nat\cup\{0\}$. Then $\mu$ has no pieces with $d+1$ independent Alberti representations if and only if, for every $m\in\nat$ and every Lipschitz map $F\colon X \to \reals^m$, there exists an $\tilde{A}(F,d)$ set $E\subset X$ with
    \begin{equation*}
        \mu(X\setminus E)=0.
    \end{equation*}
\end{theorem}

Furthermore, the notion of $\tilde{A}$ sets can be characterised via weak tangent fields. This is the content of \cite[Lemma 2.8 and Lemma 2.9]{B} which we present in the form of a single theorem.
\begin{theorem}\label{T:afd-wtf}
    Let $S\subset X$, and $F\colon X \to V$ be a Lipschitz function. Then $S$ admits a $d$-dimensional tangent field with respect to $F$ if and only if $S$ is an $\tilde A(F,d)$ set.
\end{theorem}
\begin{proof}
    In case $V=\reals^m$, this is precisely the statement of Lemmas 2.8 and 2.9 of \cite{B}. For a general $V$, the statement follows from the fact that $V$ is linearly biLispchitz to $\reals^m$ for $m=\dim V$.
\end{proof}

Combination of the two above results relates Alberti representations to weak tangent fields.
\begin{theorem}\label{T:dimA-iff-wtf}
    Let $\mu$ be a $\sigma$-finite measure on $X$ and let $d\in\nat\cup\{0\}$. Then $\dim_T(\mu)\leq d$ if and only if there exists Borel set $E\subset X$ admitting a $d$-dimensional weak tangent field with
    \begin{equation*}
        \mu(X\setminus E)=0.
    \end{equation*}
    In particular,
    \begin{equation*}
        \dim_T(\mu)=\sup\{d\in\nat\cup\{0\}: \mu\;\text{has a piece with $d$ independent Alberti representations}\}
    \end{equation*}
\end{theorem}

It should also be noted that \emph{it does not follow} from the preceding theorem, that for a Borel subset $E$ of a complete separable metric space $X$, $E$ admits a $d$-dimensional weak tangent field if and only if every $\sigma$-finite Borel measure $\mu$ on $X$ satisfies $\dim_T(\mu)\leq d$. In fact, the former statement implies the latter, but the latter statement only implies that for any one measure $\mu$ on $E$ one can choose a $\mu$-null set $N$ such that $E\setminus N$ has a $d$-dimensional weak tangent field.

The following result shows that $\dim_H \geq \dim_T$. The same conclusion may also be obtained from \cite{ArRaDimension}
by entirely different methods. Moreover, the deep result of DePhilippis and Rindler \cite[Corollary 1.12]{DeRindler} may be used to assert an even stronger result (see Remark \ref{R:AC}).
\begin{lemma}[Theorem 5.3 \cite{BKO}]\label{L:DeRindler}
    Suppose $X$ is a Polish metric space and $\mu$ a $\sigma$-finite Borel measure on $X$. If $\mu$ has $d$-independent Alberti representations, then $\dim_H \mu\geq d$.
\end{lemma}

We conclude the preliminary section by explicitly stating a fairly technical result from the work of Bate \cite{B} concerning a particular construction of a map that ``squashes'' a set in a direction in which the set does not have any curves. This is a basic building block of the argument we carry out in Section \ref{S:Squashing}.

\begin{proposition}[\cite{B} Proposition 3.5]
  \label{prop:basic-perturbation}
  Let $F\colon X \to V$ be Lipschitz, $0<\theta<1$ and $W \subset V$ a subspace.
  Suppose that a compact $S\subset X$ satisfies $\H^1(\gamma\cap S)=0$ for every $\gamma\in \Gamma(X)$ in the $F$-direction of $E(W,\theta)$.
  Further, suppose that $T\colon V\to \reals$ is linear with $W\leq \ker T$.
  Then for any $\varepsilon>0$ there exists a Lipschitz $f\colon X \to \reals$ and a $\rho>0$ such that:
  \begin{itemize}
  \item For every $y,z\in X$,
    \begin{equation*}
      |f(y)-f(z)| \leq |T(F(y)-F(z))| + 3(1-\theta)\|T\| \tLip(F) \metric(y,z);
    \end{equation*}
  \item For every $x\in X$,
    \begin{equation*}
      |T(F(x))-f(x)| < \varepsilon;
    \end{equation*}
  \item For every $x\in S$ and $y,z \in B(x,\rho) \cap S$,
    \begin{equation*}
      |f(y)-f(z)| \leq 3(1-\theta) \|T\|\tLip(F) \metric(y,z).
    \end{equation*}
  \end{itemize}
\end{proposition}

\section{Perturbations of measures to sets of low dimension}\label{S:Squashing}
The purpose of this section is to prove the density statement in Theorem \ref{T:intro-squashing} (recall that when showing a set is residual, one always needs to prove \emph{density} and \emph{openness} of certain sets). 
That is, given a measure $\mu$ on a metric space $X$, we show that any map on $X$ can be approximated with maps that squash a very large $\mu$-proportion of $X$ into a small set. 

The first subsection concerns a construction of a map that perturbs a measure on the real line. In Subsection \ref{SS:into-real-line}, this is combined with Proposition \ref{prop:basic-perturbation} to obtain mappings of measures on general metric spaces. Explicitly, we use the aforementioned proposition to obtain a map $f$ from the metric space $X$ into $\reals$ and then perturb the measure $f_\#\mu$ via a map $g\colon \reals \to \reals$. The particular properties of the two constructions result in a map $g\circ f\colon X \to \reals$ which ``squashes'' most of $\mu$ and is $1$-Lipschitz.

In Subsection \ref{SS:vector-valued-local}, we work with maps into finite dimensional Banach spaces ``coordinate by coordinate''.
By using the results of the previous subsections, we obtain maps that squash a small piece of $\mu$ (i.e.~we obtain \emph{local} results). Finally, in Subsection \ref{SS:vector-valued-global} we use a particular ``glueing'' technique from \cite{B} to glue the local results into global ones.

It is essential that the map $g\circ f$ obtained in the first two subsections is 1-Lipschitz.
This is to ensure that function constructed in the latter subsections lies in $\tLip_1(X,\reals^m)$.
A very subtle point of our construction is that the map $g$ is \textbf{not} 1-Lipschitz.
Indeed, for sometime we sought a proof for which $g$ is 1-Lipschitz, so that $g\circ f$ is automatically 1-Lipschitz.
However, it is not clear to us how to obtain such a $g$.
Instead, in the construction presented below, we permit $g$ to expand the distance between points, but only when this stretching is compensated for by $f$.

\subsection{Concentrating measure in the real line}\label{SS:real-line}

We begin with a construction that concentrates measure on the real line into a finite set, with a function that approximates the identity and which is required to be $1$-Lipschitz on scales which are not too small.
On the other hand, we allow the Lipschitz constant to be very large locally, but we require that the size of the neighbourhoods on which this happens is arbitrarily small.

Suppose $R\in (0,\infty)$, $N\in\nat, N\geq 2$, $t_0\in\reals$ are given. Define, for $k\in \mathbb{Z}$, the intervals
\begin{equation*}
    \begin{split}
        &I_k=I_k(R,N,t_0)=[t_0+kR, t_0+kR+\tfrac{R}{N}),\\
        &J_k=J_k(R,N,t_0)=[t_0+kR+\tfrac{R}{N},t_0+(k+1)R).
    \end{split}
\end{equation*}
These intervals form a partition of $\reals$ as we vary $k\in \mathbb{Z}$. Furthermore, if $R$ and $N$ are fixed, then the intervals
\begin{equation*}
    I_k(R,N,i\tfrac{R}{N}),\quad i\in\{0,\dots,N-1\}, k\in \mathbb{Z}
\end{equation*}
form a partition of $\reals$. Consequently,
    given any finite Borel measure $\mu$ on $\reals$ and any $R\in (0,\infty)$, $N\geq 2$, there exists $t_0\in [0,R)$ such that
    \begin{equation}\label{E:RL-constr-total-measure-est}
        \mu(\bigcup_{k} I_k(R,N,t_0))\leq \frac{1}{N}\mu(\reals).
    \end{equation}
We define
\begin{equation*}
    \varphi_{R,N,t_0}(t)=\begin{cases}
				N
					& \text{if $t\in I_k(R,N,t_0)$ for some $k\in\mathbb{Z}$},
					\\
				0
					& \text{otherwise},
			\end{cases}
\end{equation*}
and we let
\begin{equation*}
    g_{R,N,t_0}(t)=\int_{t_0}^t \varphi(s)\diff s.
\end{equation*}
We will write $\varphi$ instead of $\varphi_{R,N,t_0}$ and $g$ instead of $g_{R,N,t_0}$ whenever the relevant parameters are clear from the context.
Note that $g$ maps $\bigcup_k J_k$ into a countable set.

\begin{lemma}\label{L:RL-constr-sup-norm-est}
    For any $t_0\in \reals$, $R\in (0,\infty)$, $N\in \nat$ and all $t\in \reals$ it holds that
    \begin{equation}\label{E:RL-constr-identity-estimate}
        |g(t)-t|\leq |t_0|+R.
    \end{equation}
\end{lemma}
\begin{proof}
    For each $k\in\integers$, $k\geq 0$ we have
    \begin{equation*}
        I_k\cup J_k= [t_0+kR, t_0+ (k+1)R)
    \end{equation*}
    and for $t\in I_k\cup J_k$, we have,
    by definition,
    \begin{equation*}
         kR=\int_{t_0}^{t_0 + kR}\varphi(s)\diff s \leq g(t) \leq
         \int_{t_0}^{t_0 + (k+1)R}\varphi(s)\diff s=(k+1)R.
    \end{equation*}
    That is, $g(t)\in [kR, (k+1)R]$, which implies \eqref{E:RL-constr-identity-estimate}. The case $k<0$ is proven analogously.
\end{proof}

\begin{observation}\label{O:RL-constr-lip}
    For any $t_0\in \reals$, $R\in (0,\infty)$ and $N\in \nat$, the function $g$ is $N$-Lipschitz on $\reals$.
\end{observation}

\begin{lemma}\label{L:RL-constr-large-scale-constant-est}
    For any $t_0\in \reals$, $R\in (0,\infty)$, $N\in \nat$ and any $s,t\in \reals$ with $|s-t|\geq 2R$, it holds that
    \begin{equation}\label{E:RL-constr-lip-est-main}
        \frac{|g(s)-g(t)|}{|s-t|}\leq \frac{\lfloor \frac{1}{R}|s-t|\rfloor+1}{\lfloor \frac{1}{R}|s-t|\rfloor-1}
    \end{equation}
\end{lemma}
\begin{proof}
    Fix $n\in \nat$, $n\geq 2$ and let $k_0\in\integers$, $k_1\in\integers$ satisfy $k_1-k_0=n$. Suppose $s\in I_{k_0}\cup J_{k_0}$, $t\in I_{k_1} \cup J_{k_1}$.
    Now we consider what happens at the endpoints of the intervals.
    We have $s<t$ and
    \begin{equation*}
        s\leq t_0+(k_0+1)R \quad \text{and} \quad t\geq t_0+k_1R,   
    \end{equation*}
    altogether:
    \begin{equation*}
        |s-t|\geq |t_0+k_1R - t_0 + (k_0+1)R| = nR-R.
    \end{equation*}
    Similarly, using the fact that $g$ is non-decreasing,
    \begin{equation*}
        |g(s) - g(t)| \leq |g(t_0+(k_1+1)R)-g(t_0+k_0 R)| = nR+R.
    \end{equation*}
    Thus,
    \begin{equation}\label{E:RL-constr-glob-lip-est}
        \frac{|g(s)-g(t)|}{|s-t|}
        \leq \frac{nR+R}{nR-R}
        =\frac{n+1}{n-1}.
    \end{equation}
    For $t\in\reals$, define $k(t)\in\integers$ to be the unique integer for which $t\in I_{k(t)}\cup J_{k(t)}$.
    For $t,s\in \reals$, let $n(s,t)=|k(t)-k(s)|$. Observe that if $|t-s|\geq 2R$, then $n(s,t)\geq 2$. Thus, \eqref{E:RL-constr-glob-lip-est} implies that, whenever $|s-r|\geq 2R$,
    \begin{equation}\label{E:RL-constr-est-with-nst}
        \frac{|g(s)-g(t)|}{|s-t|}\leq \frac{n(s,t)+1}{n(s,t)-1}.
    \end{equation}
    As the length of each $I_k\cup J_k$ is $R$, we have
    \begin{equation*}
        n(s,t)\geq \lfloor \tfrac{1}{R}|s-t|\rfloor.
    \end{equation*}
    Thus, by \eqref{E:RL-constr-est-with-nst}, using also the fact that $n\mapsto \frac{n+1}{n-1}$ is decreasing for $n\geq 2$, \eqref{E:RL-constr-lip-est-main} holds.
\end{proof}

\begin{corollary}\label{C:RL-constr-faraway-lip-lim}
    For any fixed $N\in \nat$ with $N\geq 2$ and $r>0$, one has
    \begin{equation*}
        \lim_{R\to 0_+} \sup_{|s-t|\geq r} \frac{|g_{R,N,t_0}(s)-g_{R,N,t_0}(t)|}{|s-t|}=1,
    \end{equation*}
    the expression being independent of $t_0\in\reals$.
\end{corollary}
\begin{proof}
    The inequality ``$\leq$'' follows from Lemma \ref{L:RL-constr-large-scale-constant-est}. For the converse, suppose $R\leq r$ and consider $s=t_0$, $t=t_0+R$. Then $g(s)=0$ and $g(t)=R$, whence 
    \begin{equation*}
        \sup_{|s-t|\geq r} \frac{|g_{R,N,t_0}(s)-g_{R,N,t_0}(t)|}{|s-t|} \geq 1.
    \end{equation*}
\end{proof}

\begin{theorem}\label{T:construction-in-real-line}
    For any $\eta>0$, there exists $L\in[1,\infty)$ such that the following holds. Given any $D>0$, $r>0$, any Borel probability measure $\mu$ on $[-D,D]$ and any $\varepsilon>0$, there exists a map $h\colon [-D,D] \to\reals$ such that
    \begin{enumerate}
        \item\label{Enum:RL-C-1} $\lVert h- \Id \rVert_{\infty}< \varepsilon$,
        \item\label{Enum:RL-C-2} $h$ is $L$-Lipschitz,
        \item\label{Enum:RL-C-3} $|h(s)-h(t)|\leq |s-t|$ whenever $s,t\in [-D,D]$ satisfy $|s-t|\geq r$,
        \item\label{Enum:RL-C-4} there is some compact set $E\subset [-D,D]$ such that $\mu(E)\geq 1- \eta$ and $h(E)$ is finite.
    \end{enumerate}
\end{theorem}
\begin{proof}
    Find $N\in\nat$ such that $\frac{1}{N}\leq \eta$ and
    let $L=N$.
    Given any $\varepsilon>0$, if $|g(t)-t|< \frac{1}{2}\varepsilon$ for $t\in [-D,D]$, then
    \begin{equation}\label{E:RL-contrs-lambda-est}
        |(1-\frac{\varepsilon}{2})g(t) - t|< \varepsilon \quad \text{for all $t\in[-D,D]$.}
    \end{equation}
    By Lemma \ref{C:RL-constr-faraway-lip-lim}, there is some $R>0$ such that
    \begin{equation}\label{E:RL-constr-lip-est-lambda}
        \sup_{|s-t|\geq r} \frac{|g_{R,N,t_0}(s)-g_{R,N,t_0}(t)|}{|s-t|} < \frac{1}{1-\frac{\varepsilon}{2}},
    \end{equation}
    for any $t_0\in \reals$. Let $R>0$ be such that \eqref{E:RL-constr-lip-est-lambda} holds and $2R<\frac{1}{2}\varepsilon$.
    Let $t_0\in [0,R)$ be such that \eqref{E:RL-constr-total-measure-est} holds. Denote $g=g_{R,N,t_0}$ and $h=(1-\frac{\varepsilon}{2}) g_{|[-D,D]}$. Using Lemma \ref{L:RL-constr-sup-norm-est}, we have
    \begin{equation*}
        |g(t) - t |< 2R < \frac{1}{2}\varepsilon \quad\text{for $t\in[-D,D]$.}
    \end{equation*}
    so that via \eqref{E:RL-contrs-lambda-est} we have
    \begin{equation*}
        |h(t)- t|< \varepsilon \quad\text{for all $t\in [-D,D]$,}
    \end{equation*}
   which is \ref{Enum:RL-C-1}.

    Using the estimate \eqref{E:RL-constr-total-measure-est}, we have for $\tilde{E}=[-D,D]\setminus \bigcup_{k\in \integers} I_k(R,N,t_0)$ the estimate
    \begin{equation*}
        \mu(\tilde{E})\geq 1- \frac{1}{N}\geq 1-\eta.
    \end{equation*}
    By definition of $g$, $g(\tilde{E})$ is finite, and therefore also $h(\tilde{E})$ is finite.
    While $\tilde{E}$ might not be compact, we can simply take $E=h^{-1}(h(\tilde{E}))$, which is closed (as $h(E)$ is finite and thus closed) and bounded, thus compact. 
    Therefore, we have asserted \ref{Enum:RL-C-4}. 
    
    The fact that \ref{Enum:RL-C-2} holds is an immediate consequence of Observation \ref{O:RL-constr-lip} and of the definition of $h$.
    It remains to show \ref{Enum:RL-C-3}, which is a consequence of \eqref{E:RL-constr-lip-est-lambda} and the fact that $h=(1-\frac{\varepsilon}{2}) g_{R,N,t_0}$.
\end{proof}

\subsection{Construction of a real-valued perturbation that concentrates measure into finite sets}\label{SS:into-real-line}

In the previous subsection we constructed a map which, while $1$-Lipschitz on scales which are not too small, can have a very large Lipschitz constant due to points at small scales. In order to utilize this construction, we combine it with Proposition \ref{prop:basic-perturbation}.
We next show that that the good properties of the aforementioned proposition can be preserved after extending the relevant maps.

\begin{lemma}\label{L:RVP-extension}
    Let $X$ be a metric space, $\delta>0$ and $\varphi\colon X \to \reals$ a Lipschitz map. For every $\varepsilon>0$, there is $\varepsilon_0>0$ such that for any set $S\subset X$ and $g\colon S \to \reals$ satisfying
    \begin{equation}\label{E:RVP-ext-est}
        |g(y)-g(z)|\leq |\varphi(y)-\varphi(z)|+\delta d(y,z),
    \end{equation}
    for all $y,z\in S$ and
    \begin{equation}\label{E:RVP-ext-nearness-est}
        |g(x)-\varphi(x)|< \varepsilon_0,
    \end{equation}
    for all $x\in S$,
    there exists an extension of $g$ (denoted again by $g$) to $X$ such that
    \eqref{E:RVP-ext-est} holds for all $y,z\in X$ and \eqref{E:RVP-ext-nearness-est} holds for all $x\in X$ upon replacing $\varepsilon_0$ with $\varepsilon$.
\end{lemma}
\begin{proof}
    Define 
    \begin{equation*}
        \mathfrak{d}(x,y)=|\varphi(x)-\varphi(y)|+\delta d(x,y),\quad \text{for $x,y\in X$.}
    \end{equation*}
    The map $\mathfrak{d}$ is a metric on $X$.
    There exist $\Delta>0$ and $\varepsilon>\varepsilon_0>0$ such that both
    \begin{equation}
        \label{A:RVP-epsilon0-est1}\varepsilon_0\leq \delta \Delta
    \end{equation}
    and
    \begin{align}
        \label{A:RVP-epsilon0-est2}\Delta(2\tLip(\varphi)+\delta)+\varepsilon_0<\varepsilon
    \end{align}
    hold.
    Now let $g$ satisfy \eqref{E:RVP-ext-est} and \eqref{E:RVP-ext-nearness-est} and define
    \begin{equation*}
        G(x)=\begin{cases}
				g(x)
					& \text{if $x\in S$},
					\\
				\varphi(x)
					& \text{if $x\in X\setminus B(S,\Delta)$}.
			\end{cases}
    \end{equation*}
    Take $x\in S$ and $y\in X\setminus B(S,\Delta)$. Then, as \eqref{A:RVP-epsilon0-est1} holds and $\Delta\leq d(x,y)$, \eqref{E:RVP-ext-nearness-est} gives
    \begin{equation*}
        |G(x)-G(y)|\leq |g(x)-\varphi(x)|+|\varphi(x)-\varphi(y)|
        \leq \varepsilon_0 + |\varphi(x)-\varphi(y)|\leq \delta d(x,y)+|\varphi(x)-\varphi(y)|.
    \end{equation*}
    In other words, $G$ is $1$-Lipschitz on its domain equipped with the metric $\mathfrak{d}$. Using McShane's extension theorem, we may extend $G$ to $X$ (the extension being denoted again by $G$) in such a way that
    \begin{equation*}
        |G(x)-G(y)|\leq |\varphi(x)-\varphi(y)|+\delta \metric(x,y)\quad \text{for all $x,y\in X$}.
    \end{equation*}
    Note that $G$ is also an extension of $g$ and so it suffices to show that 
    \begin{equation*}
        |G(x)-\varphi(x)|< \varepsilon \quad \text{for all $x\in X$.}
    \end{equation*}
    If $x\in S$, this follows by $\varepsilon> \varepsilon_0$, while if $x\in X\setminus B(S,\Delta)$, this is immediate since $G(x)=\varphi(x)$. Therefore, it remains to consider the case when $x\in B(S, \Delta)\setminus S$. Let $\eta>0$ be such that
    \begin{equation*}
        (\eta+\Delta)(2\tLip(f)+\delta)+\varepsilon_0<\varepsilon.
    \end{equation*}
    The existence of such an $\eta$ is guaranteed by \eqref{A:RVP-epsilon0-est2}. Find $y\in S$ such that $d(x,y)<\Delta + \eta$. Then
    \begin{equation*}
        \begin{split}
            |G(x)-\varphi(x)|&\leq |G(x)-G(y)|+|G(y)-\varphi(y)|+|\varphi(y)-\varphi(x)|\\
            &\leq \delta d(x,y)+ |\varphi(x)-\varphi(y)|+\varepsilon_0+ |\varphi(x)-\varphi(y)|\\
            &\leq \delta(\Delta + \eta) +\tLip(\varphi)(\Delta+\eta) + \varepsilon_0 +\tLip(\varphi)(\Delta+\eta)\\
            &\leq (\Delta+ \eta)(2\tLip(\varphi) + \delta)+\varepsilon_0< \varepsilon.
        \end{split}
    \end{equation*}
\end{proof}

Given two subsets $E,F$ of a metric space $(X,\metric)$ we use the notation
\begin{equation*}
    \dist(E,F)= \inf_{x\in E, y\in F} \metric(x,y).
\end{equation*}

\begin{lemma}\label{L:RVP-domain-to-target-flatness}
    Let $S$ be a compact metric space and, $\rho>0$, $\delta>0$ and suppose that $f\colon S \to \reals$ is a Lipschitz function such that for every $x,y \in S$ one has
    \begin{equation}\label{E:RVP-almost-flat-est}
        \metric(x,y)\leq \rho \implies |f(x)-f(y)|\leq \delta \metric(x,y).
    \end{equation}
    Let $K=f(S)\subset \reals$ and define
    \begin{equation*}
        \mathfrak{d}(s,t)=\dist(f^{-1}(s), f^{-1}(t))\quad \text{for $s,t\in K$.}
    \end{equation*}
    Then there exists some $r>0$ such that if $s,t\in K$ satisfy $|s-t|\leq r$, then
    \begin{equation}\label{E:RVP-pushforward-est}
        \mathfrak{d}(s,t)\geq \frac{|s-t|}{\delta}.
    \end{equation}
\end{lemma}
\begin{proof}
    Let $r>0$ be such that $\frac{\rho}{r}\geq \frac{1}{\delta}$.
    Now, assume $s,t\in K$ are such that $|s-t|\leq r$. We may assume $s\not= t$ as otherwise the inequality is satisfied. There are two possible cases. Either, $\mathfrak{d}(s,t)\leq \rho$, or not.
    In the affirmative case, since both $f^{-1}(s)$ and $f^{-1}(t)$ are compact sets, there exists some $x\in f^{-1}(s)$ and $y\in f^{-1}(t)$ such that $\metric(x,y)=\mathfrak{d}(s,t)\leq \rho$ and we are done as \eqref{E:RVP-pushforward-est} is now implied by \eqref{E:RVP-almost-flat-est}.
    On the other hand, if $\mathfrak{d}(s,t)> \rho$, then as $|s-t|\leq r$, we have
    \begin{equation*}
        \mathfrak{d}(s,t)>|s-t| \frac{\rho}{r},
    \end{equation*}
    which in combination with the estimate $\frac{\rho}{r}\geq \frac{1}{\delta}$ implies \eqref{E:RVP-pushforward-est}.
\end{proof}

We now combine the previous results with Proposition \ref{prop:basic-perturbation}.

\begin{theorem}\label{T:RVP-main}
    For every $\eta\in (0,1)$ and $\delta>0$, there exists $\theta \in (0,1)$ such that the following is true. Let $X$ be a metric space and $S\subset X$ a compact set. 
    Suppose $V$ is a finite-dimensional Banach space and
    $F\colon X \to V$ is a Lipschitz map. Suppose further that $W\subset V$ is a subspace such that
    \begin{equation*}
        \H^1(\gamma \cap S)=0 \quad \text{for every $\gamma \in \Gamma(X)$ in the $F$-direction of $E(W,\theta).$}
    \end{equation*}
    Let $T\colon V \to \reals$ be a linear map with the property that $W\subset \ker T$.
    Then, for every $\varepsilon>0$ and every Borel probability measure $\mu$ on $S$, there exist a function $g\colon X \to \reals$ and a compact set $E\subset S$ such that
    \begin{enumerate}
        \item \label{Enum:RVP-main-1} for all $y,z\in X$ it holds that
        \begin{equation*}
            |g(y)-g(z)|\leq |T(F(y)-F(z))|+\delta \lVert T \rVert \tLip(F) \metric(y,z),
        \end{equation*}
        \item \label{Enum:RVP-main-2} for every $x\in X$ it holds that
        \begin{equation*}
            |TF(x)-g(x)|< \varepsilon,
        \end{equation*}
        \item \label{Enum:RVP-main-3} $\mu(E)\geq 1- \eta$ and $g(E)$ is a finite set.
    \end{enumerate}
\end{theorem}
\begin{proof}
    Given $\eta>0$ find $L\in [1,\infty)$ from Theorem \ref{T:construction-in-real-line}.
    Now let $\theta \in (0,1)$ be such that
    \begin{equation}\label{E:delta-choice}
        L3(1-\theta)\leq \delta.
    \end{equation}
    We use Lemma \ref{L:RVP-extension}, replacing $\delta$ with $\delta \lVert T \rVert \tLip(F)$ and $\varphi$ with $T\circ F$ to find some $\varepsilon_0$ having the property specified in Lemma \ref{L:RVP-extension}.
    
    Using Proposition \ref{prop:basic-perturbation} we can find a function $f\colon X \to \reals$ and $\rho>0$ such that
    \begin{enumerate}[label={\rm(\arabic*)}]
        \item\label{Enum:app-Bate-1} for every $y,z\in X$:
        \begin{equation*}
            |f(y)-f(z)|\leq |T(F(y)-F(z))|+3(1-\theta)\lVert T \rVert \tLip(F) \metric(y,z),
        \end{equation*}
        \item\label{Enum:app-Bate-2} for every $x\in X$:
        \begin{equation*}
            |T(F(x))-f(x)|< \frac{1}{2}\varepsilon_0,
        \end{equation*}
        \item\label{Enum:app-Bate-3} for all $x\in S$ and $y,z\in B(x,\rho)\cap S$:
        \begin{equation*}
            |f(y)-f(z)|\leq 3(1-\theta)\lVert T \rVert \tLip(F)\metric(y,z).
        \end{equation*}
    \end{enumerate}
    
    Now let $K=f(S)$ and $D>0$ be such that $K\subset [-D,D]$. Let us define
    \begin{equation*}
        \mathfrak{d}(s,t)=\dist(f^{-1}(s), f^{-1}(t))\quad \text{for $s,t\in K$.}
    \end{equation*}
    Recalling Lemma \ref{L:RVP-domain-to-target-flatness}, there is some $r>0$ such that
    \begin{equation}\label{E:RVP-metric-pfwd-est}
        \mathfrak{d}(s,t)\geq \frac{|s-t|}{3(1-\theta)\lVert T \rVert\tLip(F)}\quad \text{for all $s,t\in K$ satisfying $|s-t|\leq r$.}
    \end{equation}

    Since $f_\# \mu$ is a Borel measure on $\reals$, using Theorem \ref{T:construction-in-real-line}, we find a function $h\colon [-D,D] \to \reals$ such that
    \begin{enumerate}[label={\rm(\alph*)}]
        \item\label{Enum:app-T-RLC-a} $\lVert h- \Id \rVert_{\infty}< \frac{1}{2}\varepsilon_0$
        \item\label{Enum:app-T-RLC-b} $|h(s)-h(t)|\leq L$ if $|s-t|\leq r$ and $|h(s)-h(t)|\leq 1$ if $|s-t|\geq r$,
        \item\label{Enum:app-T-RLC-c} there is some compact $\widetilde H \subset [-D,D]$
        such that $f_{\#}\mu(\widetilde H)\geq 1-\eta$ and $h(\widetilde H)$ is finite.
    \end{enumerate}
    Let $H=\widetilde H \cap K$. As $K$ is of full measure with respect to $f_{\#}\mu$, \ref{Enum:app-T-RLC-c} remains true if we replace $\widetilde H$ with $H$.

    Let $g=h\circ f_{|S} \colon S \to \reals$. Then, by \ref{Enum:app-Bate-2} and \ref{Enum:app-T-RLC-a} we have
    \begin{equation*}
        |T(F(x))-g(x)|<\varepsilon_0\quad \text{for every $x\in S$.}
    \end{equation*}
    If $|f(y)-f(z)|\leq r$, then by the definition of $\mathfrak{d}$, from the estimate \eqref{E:RVP-metric-pfwd-est}, the first half of \ref{Enum:app-T-RLC-b} and \eqref{E:delta-choice}, we have
    \begin{equation*}
        |g(y)-g(z)|\leq L3(1-\theta) \lVert T \rVert \tLip(F)d(y,z)\leq \delta \lVert T \rVert\tLip(F) d(y,z).
    \end{equation*}
    On the other hand, if $|f(y)-f(z)|\geq r$, then by the second half of \ref{Enum:app-T-RLC-b} and \ref{Enum:app-Bate-1}, we have
    \begin{equation*}
    \begin{split}
        |g(y)-g(z)|&\leq |f(y)-f(z)|\leq |T(F(y)-F(z))|+3(1-\theta)\lVert T \rVert\tLip(F) d(y,z)\\
        &\leq |T(F(y)-F(z))|+\delta\lVert T \rVert\tLip(F) d(y,z).
    \end{split}
    \end{equation*}
    Combining the above, we have
    \begin{equation*}
        |g(y)-g(z)|\leq |T(F(y)-F(z))|+\delta\lVert T \rVert\tLip(F) d(y,z) \quad \text{for all $y,z \in S$.}
    \end{equation*}
    Whence, recalling the definition of $\varepsilon_0$ and using Lemma \ref{L:RVP-extension}, $g$ admits an extension to $X$ such that \ref{Enum:RVP-main-1} and \ref{Enum:RVP-main-2} hold. Let $E=g^{-1}(H)\cap S$. Then $E$ satisfies \ref{Enum:RVP-main-3} due to our definition of $H$ and the fact that $\mu$ is supported on $S$.
\end{proof}

\subsection{Obtaining local vector-valued perturbations from the real valued perturbations}\label{SS:vector-valued-local}

In this subsection we follow, to the word, the definitions and conventions of \cite[Section 4]{B}. The approach to obtain vector-valued perturbations given well-behaved real-valued perturbations is the same as in the aforementioned paper. We omit some of the details and justifications, as they can all be found therein.

Given a finite-dimensional vector space $V$ and a basis $b_1,\dots, b_m$ of $V$ consisting of unit vectors, we denote by $b^*_i$, $i\in\{1,\dots, m\}$ the $i$-th coordinate functional:
\begin{equation*}
    b^*_i\left(\sum_{j=1}^{m}\lambda_j b_j\right)=\lambda_i.
\end{equation*}

We let $K_u$ be the minimal constant $K$ such that the inequality
\begin{equation*}
    \lVert \sum_{i=1}^m l_i b_i^*(x)b_i\rVert \leq K \lVert l \rVert_{\infty} \lVert x \rVert_V
\end{equation*}
holds for every $x\in V$ and $l\in \ell^m_\infty$. The quantity $K_u$ is always finite and $K_u=1$ if $V=\ell_p^m$ for any $p\in[1,\infty]$.

\begin{definition}\label{D:K(V,d)}
    For $V$ a finite dimensional Banach space and an integer $d\geq 0$, we let $\tilde K(V,d)$ be the least $K \geq 1$ for which the following is true:
    There exists $K_d,K_p >0$ and, for any $d$-dimensional subspace $W$ of $V$, a basis $b_1,\ldots,b_m$ of $V$ consisting of unit vectors and projections $P\colon V\to W$ and $Q\colon V\to \ker P$ such that:
    \begin{enumerate}
	  \item  $P(x)+ Q(x) = x$         for all $x\in V$;
        \item $\|b_i^*\|\leq K_p$ for each $1\leq i \leq m$;
        \item $\|P\|,\|Q\|\leq K_d$;
        \item For each $x\in V$ and $l\in \ell_\infty^m$ with $\|l\|_\infty \leq 1$,
            \begin{equation*}
                \left\|P(x)+ \sum_{i=1}^m l_i b_i^*(Q(x))b_{i}\right\| \leq K \|x\|.
            \end{equation*}
    \end{enumerate}
\end{definition}    

We remark that one always has
\begin{equation*}
    \tilde K(V,d)\leq K_u(2\sqrt{d}+1),
\end{equation*}
therefore the quantity $\tilde K(V,d)$ is always finite and in the case of $V=\ell^m_p$, the quantity is independent of the dimension $m$.

\begin{observation}[\cite{B}, Observation 4.2]
    For any $m\in \nat$ and $d\geq 0$,
    \begin{equation*}
        \tilde K(\ell^m_2,d)=1.
    \end{equation*}
\end{observation}

\begin{observation}[\cite{B}, Observation 4.3]\label{O:k=1-pun}
    Suppose $V$ admits a basis such that $K_u=1$. Then
    \begin{equation*}
        \tilde K(V,0)=1.
    \end{equation*}
    In particular,
    \begin{equation*}
        \tilde K(\ell^p_m, 0)=1,
    \end{equation*}
    for every $p\in[0,\infty]$ and $m\in\nat$.
\end{observation}

For the remainder of this subsection, unless stated otherwise, we consider $V$ to be a fixed finite-dimensional Banach space. We also fix $d\geq 0$ and a $d$-dimensional subspace $W\subset V$. We assume that $K_d, K_p>0$, $P$, $Q$, $b_1,\dots, b_m$ are as in Definition \ref{D:K(V,d)}.
We let $T_i=b_i^* \circ Q\colon V \to \reals$.
The following is essentially the main content of \cite[Lemma 4.5]{B}.

\begin{lemma}\label{L:scalar-to-vector-B}
    There is some $C_V>0$, depending only on $V$, such that the following is true.
    Suppose $X$ is a metric space and $F\colon X \to V$ is Lipschitz and $\delta>0$. 
    Suppose that for each $i=1,\dots, m$, there is a function $g_i \colon X \to \reals$ such that
    \begin{enumerate}[label={\rm(\alph*)}]
        \item\label{Enum:scalar-to-vector-a} for all $y,z\in X$ it holds that
        \begin{equation*}
            |g_i(y)-g_i(z)|\leq |T_i(F(y)-F(z))|+\delta \lVert T_i \rVert\tLip(F) \metric(y,z),
        \end{equation*}
        \item\label{Enum:scalar-to-vector-b} for every $x\in X$ it holds that
        \begin{equation*}
            |T_i F(x)-g_i(x)|< \varepsilon.
        \end{equation*}
    \end{enumerate}
    Let $\sigma\colon X \to V$ be given by
    \begin{equation*}
        \sigma(x)=P(F(x))+ \sum_{i=1}^m g_i(x)b_i.
    \end{equation*}
    Then $\sigma$ satisfies the following:
    \begin{enumerate}
        \item\label{Enum:scalar-to-vector-i} The Lipschitz constant of $\sigma$ satisfies $\tLip(\sigma)\leq \tilde{K}(V,d)\tLip(F)+\delta C_V$ and
        \item\label{Enum:scalar-to-vector-ii} for every $x\in X$,
        \begin{equation*}
            \lVert \sigma(x)-F(x)\rVert\leq m\varepsilon.
        \end{equation*}
    \end{enumerate}
\end{lemma}
\begin{proof}
    In the proof of \cite[Lemma 4.5]{B}, the function $\sigma$ is defined using a particular instance of functions $g_i$ (denoted by $f_i$ therein), which satisfy the properties \ref{Enum:scalar-to-vector-a} and \ref{Enum:scalar-to-vector-b}. The author then proves that the function $\sigma$ satisfies \ref{Enum:scalar-to-vector-i} and \ref{Enum:scalar-to-vector-ii}. However, the only properties of the particular functions $f_i$ used in the proof are \ref{Enum:scalar-to-vector-a} and \ref{Enum:scalar-to-vector-b} (possibly with renamed constants). Therefore, the proof applies in our case word for word.
\end{proof}

Using the preceding lemma, together with the real-valued perturbations obtained in Theorem \ref{T:RVP-main}, it is now easy to obtain vector-valued perturbations, which concentrate measure locally (i.e.~on a possibly small subset of the metric space $X$).

\begin{proposition}\label{P:StV-local}
    For any $\eta \in (0,1)$ and $\delta>0$, there exists $\theta\in(0,1)$, depending only on $\eta$ and $\delta$, such that the following is true. Suppose $X$ is a metric space and $S\subset X$ is compact. Assume $F\colon X \to V$ is a Lipschitz function such that
    \begin{equation*}
        \H^1(\gamma \cap S)=0 \quad \text{for all $\gamma \in \Gamma(X)$ in the $F$-direction of $E(W,\theta)$.}
    \end{equation*}
    Then, for any $\varepsilon>0$ and a Borel probability measure $\mu$ on $S$, there exists $G\colon X \to V$ and a compact set $E$ such that
    \begin{enumerate}
        \item\label{Enum:scalat-to-vector-prop-1} the Lipschitz constant of $G$ satisfies $\tLip(G)\leq \tilde{K}(V,d)\tLip(F)+\delta C_V$,
        \item\label{Enum:scalat-to-vector-prop-2} for any $x\in X$,
        \begin{equation*}
            \lVert F(x) - G(x) \rVert_{V}\leq \varepsilon,
        \end{equation*}
        \item\label{Enum:scalat-to-vector-prop-3} there is a set $\mathcal{F}\subset V$, which is a finite union of affine $d$-dimensional subspaces of $V$, such that $G(E)\subset \mathcal{F}$ and $\mu(E)\geq 1-\eta$.
    \end{enumerate}
\end{proposition}

\begin{proof}
    Since $b_1, \dots, b_d$ form a basis of $W$ it follows that $T_i=b_i^*\circ Q$ are identically zero for $i=1,\dots, d$. We let $g_i\colon X \to \reals$ be identically equal to $0$ and $E_i=X$ for such $i$'s. Let $\eta_i\in (0,1)$, $i=d+1,\dots, m$ be such that
    \begin{equation*}
        1- \sum_{i=d+1}^m \eta_i\geq 1-\eta.
    \end{equation*}
    For each of these $\eta_i$, we recall Theorem \ref{T:RVP-main} and find a function $g_i\colon X \to \reals$ and a compact set $E_i\subset X$ such that
    \begin{enumerate}[label={\rm(\roman*)$_{i}$}]
        \item\label{Enum:scalat-to-vector-prop-1-i} for all $y,z \in X$,
        \begin{equation*}
            |g_i(y)-g_i(z)|\leq |T_i(F(y)-F(z))| + \delta \lVert T_i \rVert\tLip(F) \metric(y,z),
        \end{equation*}
        \item\label{Enum:scalat-to-vector-prop-2-i} for all $x\in X$,
        \begin{equation*}
            |T_i F(x) - g_i(x)|\leq \frac{1}{m}\varepsilon,
        \end{equation*}
        \item\label{Enum:scalat-to-vector-prop-3-i} we have $\mu(E_i)\geq 1- \eta_i$ and $g_i(E_i)$ is finite.
    \end{enumerate}

    We let $E=\cap_{i=1}^m E_i$ and
    \begin{equation*}
        G(x)=P(F(x))+\sum_{i=1}^m g_i(x)b_i, \quad \text{for $x\in X$.}
    \end{equation*}
    Then, using Lemma \ref{L:scalar-to-vector-B} and \ref{Enum:scalat-to-vector-prop-1-i} and \ref{Enum:scalat-to-vector-prop-2-i}, we see that \ref{Enum:scalat-to-vector-prop-1} and \ref{Enum:scalat-to-vector-prop-2} are satisfied.

    To show \ref{Enum:scalat-to-vector-prop-3}, we first observe that, for each $i\geq d+1$,
    \begin{equation*}
        b_i^*\circ G = b_i^*\circ(PF + \sum_{j=1}^m g_j b_j)= b_i^*\circ(\sum_{j=1}^m g_j b_j)=g_i.
    \end{equation*}
    Indeed, as $b_i$'s form a basis of $V$ and $b_1, \dots, b_d$ form a basis of $W$, it follows that $P(V)= W \subset \ker b_i^*$ for $i\geq d+1$. 
    Therefore $H_i=b_i^*(G(E))$ is finite for each $i=d+1,\dots m$. Furthermore,
    \begin{equation}\label{E:scalar-to-vector-d-dim-cont}
        G(E)\subset \bigcap_{i=d+1}^m (b_i^*)^{-1}(H_i).
    \end{equation}
    Elementary linear algebra shows that for any $t_{d+1}, \dots, t_d\in \reals$,
    \begin{equation*}
        \dim\left(\bigcap_{i=d+1}^m (b_i^*)^{-1}(t_i)\right)=d.
    \end{equation*}
    Whence, for $H=\bigcup_{i=d+1}^m H_i$, the set
    \begin{equation*}
        \bigcap_{i=d+1}^m (b^*_i)^{-1}(H)
    \end{equation*}
    is a finite union of $d$-dimensional affine subspaces of $V$, which in combination with \eqref{E:scalar-to-vector-d-dim-cont} yields \ref{Enum:scalat-to-vector-prop-3}. 
\end{proof}

\begin{observation}\label{O:finiteness-subspaces}
    Suppose $V$ is a finite-dimensional Banach space and $\mathcal{F}\subset V$ is a finite union of $d$-dimensional affine subspaces. If $F\subset \mathcal{F}$ is compact, then $\H^d(F)<\infty$.
\end{observation}
\begin{proof}
    As any $d$-dimensional subspace $U$ of $V$ is biLipschitz to $\reals^d$, and since $F\cap U$ is compact, $\mathcal{H}^d(F\cap U)< \infty$. Therefore, $F$ is a finite union of sets of finite measure.
\end{proof}

\subsection{Obtaining global vector-valued perturbations from local ones}\label{SS:vector-valued-global}

\begin{theorem}\label{T:StV-main}
    Let $X$ be a complete metric space, let $S\subset X$, assume $V$ is a finite-dimensional Banach space, $L>0$ and $F\colon X \to V$ is an $L$-Lipschitz function. Let $\eta\in (0,1)$ and $\varepsilon>0$ and assume that $S$ is an $\tilde{A}(F,d)$ set. Finally, suppose $\mu$ is a Borel probability measure on $S$. Then there exists a function $G\colon X \to V$ and a compact set $E\subset S$ such that
    \begin{enumerate}
        \item\label{Enum:StV-main-1} the Lispchitz constant of $G$ satisfies $\tLip(G)\leq \tilde{K}(V,d) L$,
        \item\label{Enum:StV-main-2} for any $x\in X$,
        \begin{equation*}
            \lVert F(x) -G(x)\rVert_{V}\leq \varepsilon,
        \end{equation*}
        \item\label{Enum:StV-main-3} $\mu(E)\geq 1-\eta$ and $G(E)$ is contained in a finite union of $d$-dimensional affine subspaces of $V$.
    \end{enumerate}
\end{theorem}
\begin{proof}
    Without loss of generality, we may assume $\tLip(F)<L$. There is some $\delta>0$ such that
    \begin{equation}\label{E:StV-main-delta-choice}
        \tilde{K}(V,d)\tLip(F)+\delta C_V < \tilde{K}(V,d)L.
    \end{equation}
    Using Proposition \ref{P:StV-local}, there is some $\theta \in (0,1)$ such that the following holds. Whenever $S_0\subset X$ is compact and $W\subset V$ is a $d$-dimensional subspace such that
    \begin{equation*}
        \H^1(\gamma \cap S_0)=0 \quad \text{for all $\gamma\in \Gamma(X)$ in the $F$-direction of $E(W,\theta)$,}
    \end{equation*}
    then for every $\teps$ and any Borel finite non-trivial measure $\mu_0$ on $S_0$, there exists $G_0\colon X \to V$ and a compact set $E_0\subset S_0$ such that
    \begin{enumerate}[label={\rm(\alph*)}]
        \item\label{Enum:StV-main-a} the Lipschitz constant of $G_0$ satisfies $\tLip(G_0)\leq \tilde{K}(V,d)\tLip(F)+\delta C_V$,
        \item\label{Enum:StV-main-b} for all $x\in X$,
        \begin{equation*}
            \lVert F(x) -G_0(x) \rVert_{V}\leq \teps,
        \end{equation*}
        \item\label{Enum:StV-main-c} $\mu_0(E_0)\geq\mu_0(S_0) \sqrt{1-\eta}$ and there is some finite union $\mathcal{F}_0$ of $d$-dimensional affine subspaces of $V$ such that $G_0(E_0)\subset \mathcal{F}_0$.
    \end{enumerate}
    As $S$ is an $\tilde{A}(F,d)$ set, we can find a decomposition $S=S_1\cup \dots \cup S_M$ and $d$-dimensional subspaces $W_1, \dots, W_M\subset V$ such that
    \begin{equation*}
        \H^1(\gamma \cap S_i)=0 \quad \text{for every $\gamma \in \Gamma(X)$ in the $F$-direction of $E(W_i,\theta)$.}
    \end{equation*}
    Using the regularity of $\mu$, we find compact subsets $\tilde{S}_i \subset S_i$ such that
    \begin{equation}\label{E:StV-main-tilde-S-est}
        \mu(\bigcup_i \tilde{S}_i)\geq \sqrt{1-\eta}.
    \end{equation}
    As the sets $\tilde{S}_i$ are pairwise disjoint and compact, there is some $\rho>0$ such that the sets $B(\tilde{S}_i, \rho)$ are pairwise disjoint.
    Denote $\mu_i=\mu_{|\tilde{S}_i}$.

    Given $\teps>0$ and $i=1,\dots, M$, if $\mu_i$ is trivial, let $E_i=\emptyset$, $\mathcal{F}_i=\emptyset$ and $G_i=F$. 
    If $\mu_i$ is not trivial, then by our choice of $\theta$, there exists a function $G_i\colon X \to V$ and a compact set $E_i \subset \tilde{S_i}$ such that \ref{Enum:StV-main-a}, \ref{Enum:StV-main-b} and \ref{Enum:StV-main-c} above hold with $\mu_0=\mu_i$, $G_0=G_i$ 
    and $E_0=E_i$ and the respective finite unions of $d$-dimensional affine subspaces denoted by $\mathcal{F}_i$.

    Let $\mathcal{F}=\bigcup_i \mathcal{F}_i$ and observe that $\mathcal{F}$ is a finite union of $d$-dimensional affine subspaces of $V$. Further, let $E= \bigcup_i E_i$. Denote
    \begin{equation*}
        L(\teps)=\tilde{K}(V,d)\tLip(F)+\delta C_V + \frac{2\teps}{\rho}.
    \end{equation*}
    By \cite[Lemma 4.6]{B}, there exists a function $G\colon X \to V$ (depending in particular on $\teps$) such that for each $i$, one has $G_{|\tilde{S}_i}=G_{i|\tilde{S}_i}$, $\tLip(G)\leq L(\teps)$ and for every $x\in X$, one has
    \begin{equation*}
        \lVert F(x)- G(x) \rVert_{V}\leq \teps.
    \end{equation*}
    Recalling \eqref{E:StV-main-delta-choice}, there exists some $\teps\in (0,\varepsilon)$ such that $G$ satisfies \ref{Enum:StV-main-1} and \ref{Enum:StV-main-2}. Moreover, by definition of $E$ and $\mathcal{F}$ and since $G_{|\tilde{S_i}}=G_{i{|\tilde{S_i}}}$, we have
    \begin{equation*}
        G(E)= G(\bigcup_i E_i)= \bigcup_i G(E_i)\subset \bigcup_i \mathcal{F}_i=\mathcal{F},
    \end{equation*}
    whence $G(E)$ is a subset of $\mathcal{F}$. Finally, from \ref{Enum:StV-main-c} and \eqref{E:StV-main-tilde-S-est}, we have
    \begin{equation*}
        \mu(E) = \sum_i \mu(E_i) \geq \sum_{i} \sqrt{1-\eta} \mu(\tilde{S}_i) \geq \sqrt{1-\eta}\sqrt{1-\eta}=1-\eta.
    \end{equation*}
    Thus, \ref{Enum:StV-main-3} is also satisfied.
\end{proof}

\section{Residuality with Euclidian targets}\label{S:residuality-Euclidean}

Equipped with Theorem \ref{T:StV-main}, the final ingredient necessary to carry out the proof of Theorem \ref{T:intro-squashing} lies in showing that certain sets of maps are open. The basic tool to do this is the following result concerning Hausdorff content.

\begin{lemma}\label{L:semicont}
    Suppose $S$ is a compact metric space, $Y$ is a metric space and $s\in (0,\infty)$. The functional
    \begin{equation*}
        f\mapsto \H^s_\infty(f(X))
    \end{equation*}
    is upper semi-continuous on the space $\tLip_L(X,Y)$ for any $L\in [0,\infty)$.
\end{lemma}
\begin{proof}
    This is an equivalent way of stating \cite[Lemma 5.3]{B}.
\end{proof}

\begin{theorem}
    Suppose $X$ is a complete metric space and $S\subset X$.
    Suppose $d\in \nat$ and let $V$ be a finite-dimensional Banach space for which $\tilde{K}(V,d)=1$.
    Let $\mu$ be a $\sigma$-finite Borel measure on $S$ for which there exists a set $N\subset S$ with $\mu(N)=0$ such that $S\setminus N$ is $\sigma$-compact and an $\tilde{A}(F,d)$ set for every Lipschitz $F\colon X\to V$.
    Then, for any $L\in [0,\infty)$, the set
    \begin{equation*}
        \{f\in \tLip_L(X,V): \dim_H(f_\# \mu)\leq d\}
    \end{equation*}
    is residual and $G_\delta$ in $\tLip_L(X,V)$.
\end{theorem}
\begin{proof}
    We begin by proving the assertion assuming $S$ is a 
    compact $\tilde{A}(F,d)$ set and $\mu$ is a Borel probability measure on $S$, in which case we may simply take $N=\emptyset$. Let us write
    \begin{equation*}
        \mathcal{G}=\{f\in \tLip_L(X,V): \dim_H(f_\# \mu)\leq d\}
    \end{equation*}
    and
    \begin{equation*}
        \mathcal{G}_{\delta, \eta, \varepsilon}=\{f\in \tLip_L(X,V): \exists E \subset S\;\text{compact and such that,}\; \mu(E)\geq 1-\eta,\; \H^{d+\delta}_\infty(f(E)) < \varepsilon\},
    \end{equation*}
    for $\delta, \varepsilon>0$ and $\eta \in (0,\infty)$.
    Firstly, we shall assert that
    \begin{equation*}
        \mathcal{G}=\bigcap_{\varepsilon, \delta>0, \eta\in (0,1)} \mathcal{G}_{\delta, \eta, \varepsilon}.
    \end{equation*}
    The inclusion from left to right is immediate, we prove the converse inclusion. To that end, let $\delta>0$ and assume $f\in \mathcal{G}_{\delta, \eta, \varepsilon}$ for every $\eta\in (0,1)$ and $\varepsilon>0$. Then, by definition, for every $n,i,k\in \nat$, there exists a compact set $E(n,i,k)\subset S$ such that
    \begin{enumerate}
        \item\label{Enum:global-main-i} $\mu(S\setminus E(n,i,k))< \frac{1}{i}2^{-n}$,
        \item\label{Enum:global-main-ii} $\H^{d+\delta}_\infty(f(E(n,i,k)))< 2^{-i} \frac{1}{k}$.
    \end{enumerate}
    We let
    \begin{equation*}
        E_\delta=\bigcap_{k\in\nat} \bigcup_{i\in\nat} \bigcap_{n\in \nat} E(n,i,k).
    \end{equation*}
    Then \ref{Enum:global-main-i} and \ref{Enum:global-main-ii} imply 
    \begin{enumerate}[label={\rm(\alph*)}]
        \item $\mu(S\setminus E_\delta)=0$ and
        \item $\H^{d+\delta}_\infty(f(E_\delta))=0$.
    \end{enumerate}

    It follows that $\H^{d+\delta}(f(E_\delta))=0$, whence also $\dim_H(f(E_\delta))\leq d+\delta$. Taking
    \begin{equation*}
        E=\bigcap_{n\in \nat} E_{\frac{1}{n}},
    \end{equation*}
    yields $\mu(S\setminus E)=0$ and $\dim_H(f(E))\leq d$. In other words, $\dim_H(f_\# \mu)\leq d$ and hence $f\in \mathcal{G}$.

    Next, we show that for every $\delta, \varepsilon>0$, $\eta \in (0,1)$, the set $\mathcal{G}_{\delta, \eta, \varepsilon}$ is open and dense in $\tLip_L(X,V)$.
    Firstly, density. Fix $\delta>0$ and suppose $f\in \tLip_L(X,V)$. Since $S$ is an $\tilde{A}(f,d)$ set and $\tilde{K}(V,d)=1$, for any $\Delta>0$, we have from Theorem \ref{T:StV-main} (and Observation \ref{O:finiteness-subspaces}) a function $g\in \tLip_L(X,V)$ and a compact set $E\subset S$ such that $\mu(E)\geq 1-\eta$,
    \begin{equation*}
        \H^d(g(E))< \infty \quad \text{and}\quad \lVert f-g \rVert_{\infty}< \Delta.
    \end{equation*}
    This implies that $g\in \mathcal{G}_{\delta,\eta,\varepsilon}$ for every $\varepsilon>0$ and $\delta>0$. As $\Delta$ is arbitrary, this shows density of $\mathcal{G}_{\delta,\eta,\varepsilon}$ in $\tLip_L(X,V)$.

    To show openness, we write $\mathcal{K}(S)$ for the family of compact subsets of $S$ and
    \begin{equation*}
        \mathcal{G}_{\delta, \eta, \varepsilon}=\bigcup_{E\in \mathcal{K}(X); \mu(E)\geq 1-\eta}\{f\in \tLip_L(X,V): \H^{d+\delta}_\infty(f(E)) < \varepsilon\}.
    \end{equation*}
    Each of the sets inside the union on the right hand side is open due to Lemma \ref{L:semicont}. Therefore, the set $\mathcal{G}_{\delta, \eta, \varepsilon}$ is open.
    Since we may write
    \begin{equation*}
        \mathcal{G}=\bigcap_{\delta, \varepsilon>0, \eta\in(0,1)} \mathcal{G}_{\delta, \eta, \varepsilon}=\bigcap_{n,k,i\in \nat} \mathcal{G}_{\frac{1}{n}, \frac{1}{k}, \frac{1}{i}},
    \end{equation*}
    the set $\mathcal{G}$ is a countable intersection of open dense sets.

    In the general case, unless $\mu$ is trivial, in which case there is nothing to prove, it follows from  our assumptions, that there exists a set $ N\subset \tilde N \subset S$ and \emph{compact} $\tilde{A}(F,d)$ sets $S_1, S_2,\dots\subset S$ such that 
    \begin{enumerate}[label={\rm(\arabic*)}]
        \item $\mu(\tilde N)=0$,
        \item $0<\mu(S_i)<\infty$ and
        \item $S\setminus \tilde{N}= \bigcup_i S_i$.
    \end{enumerate}
    Denote $\mu_i= \frac{1}{\mu(S_i)}\mu_{|S_i}$. As each $S_i$ is compact and each $\mu_i$ is a probability measure on $S_i$, we have already shown that the sets
    \begin{equation*}
        \mathcal{F}_i=\{f\in \tLip_L(X,V): \dim_H(f_\# \mu_i)\leq d\}
    \end{equation*}
    are residual $G_\delta$ and thus, also the set
    \begin{equation*}
        \mathcal{F}=\bigcap_i \mathcal{F}_i
    \end{equation*}
    is residual $G_\delta$. It remains to show
    \begin{equation*}
        \{f\in \tLip_L(X,V): \dim_H(f_\# \mu)\leq d\}=\mathcal{F}.
    \end{equation*}
    The inclusion from left to right is immediate. Suppose $f\in \mathcal{F}$. Then, for each $i\in \nat$, there is a $E_i\subset S_i$ with $\mu(S_i\setminus E_i)=0$ and such that
    \begin{equation}\label{E:dim-leq-d}
        \dim_H(f(E_i))\leq d.
    \end{equation}
    Let $E=\bigcup_i E_i$. Then we have $f(E)=\bigcup_i f(E_i)$ and so by \eqref{E:dim-leq-d} we have
    \begin{equation*}
        \dim_H(f(E))= \dim_H(\bigcup f(E_i))\leq d,
    \end{equation*}
    showing that $\dim_H(f_\# \mu)\leq d$.
\end{proof}
The following is a corollary of the previous result combined with the characterisation in Theorem \ref{T:afd-wtf}.
\begin{corollary}\label{C:1}
    Suppose $X$ is a complete metric space and $S\subset X$.
    Suppose $d\in \nat$ and let $V$ be a finite-dimensional Banach space for which $\tilde{K}(V,d)=1$.
    Let $d\in \nat$ and let $\mu$ be a $\sigma$-finite Borel measure on $S$ for which there exists a set $N\subset S$ with $\mu(N)=0$ such that $S\setminus N$ is $\sigma$-compact and admits a $d$-dimensional weak tangent field with respect to every Lipschitz $F\colon X\to V$.
    Then, for any $L\in [0,\infty)$, the set
    \begin{equation*}
        \{f\in \tLip_L(X,V): \dim_H(f_\# \mu)\leq d\}
    \end{equation*}
    is residual and $G_\delta$ in $\tLip_L(X,V)$.
\end{corollary}

We continue by collecting corollaries in a slightly simpler form for the space $V=\reals^m$, recalling in particular that $\tilde K(\reals^m, d)=1$. 

\begin{corollary}\label{C:main-euclidian}
    Suppose $X$ is a complete metric space, $d\in \nat$ and $S\subset X$ is separable and admits a $d$-dimensional weak tangent field with respect to every Lipschitz $F\colon X\to \reals^m$. Then, for every $\sigma$-finite Borel measure $\mu$ on $S$, the set
    \begin{equation*}
        \{f\in \tLip_L(X,\reals^m): \dim_H(f_\# \mu)\leq d\}
    \end{equation*}
    is residual in $\tLip_L(X,\reals^m)$.
\end{corollary}
\begin{proof}
    If $S$ is Polish, then $\mu$ is supported on a $\sigma$-compact set and so the statement follows from Corollary \ref{C:1}. Since $S$ is separable and $X$ is complete, the closure of $S$ in $X$, $\overline{S}$, is Polish. So, on defining the Borel measure $\tilde{\mu}(E)=\mu(E\cap S)$ for $E\subset \overline{S}$, the set 
    \begin{equation*}
        \{f\in \tLip_L(X,\reals^m): \dim_H(f_\# \tilde{\mu})\leq d\}
    \end{equation*}
    is residual in $\tLip_L(X,\reals^m)$.
    As $\dim_H(f_\# \tilde{\mu})\leq d$ is equivalent to $\dim_H(f_\# \mu)\leq d$, the statement follows.
\end{proof}

Theorems \ref{T:intro-main-1d} and \ref{T:intro-squashing} now immediately follow from the preceding Corollary \ref{C:main-euclidian}.
We continue by stating a result for measures directly, as opposed to quantifying over all measures on a given set. This statement is what we are able to obtain a converse to, in Section \ref{S:stability}, for Euclidean ambient spaces.

\begin{corollary}\label{C:main-euclidean-measure}
    Suppose $X$ is a complete separable metric space, $d\in \nat$ and $\mu$ is a Borel measure on $X$ with $\dim_T \mu\leq d$. Then the set
    \begin{equation*}
        \{f\in \tLip_L(X,\reals^m): \dim_H(f_\#\mu)\leq d\}
    \end{equation*}
    is residual in $\tLip_L(X, \reals^m)$.
\end{corollary}
\begin{proof}
    By Theorem \ref{T:dimA-iff-wtf}, there exists a Borel set $S\subset X$, admitting a $d$-dimensional tangent field such that $\mu$ is supported on $S$. Thus, the result follows from Corollary \ref{C:main-euclidian}. 
\end{proof}

Finally, when $d=0$, it is not necessary to assume that the target vector space is Euclidean due to Observation \ref{O:k=1-pun}. Indeed, it is enough to assume that the target space admits a basis for which the constant $K_u=1$, which is in particular true for $\ell^p_m$ spaces. 

\begin{corollary}
    Suppose $X$ is a complete metric space and $S\subset X$ is separable and purely 1-unrectifiable. Then for every Borel finite measure $\mu$ supported on $S$, $m\in\nat$, $p\in[0,1]$ and $L\in[0,\infty)$, the set
    \begin{equation*}
        \{f\in \tLip_L(X,\ell^p_m): \dim_H(f_\#\mu)=0\}
    \end{equation*}
    is residual in $\tLip_L(X,\ell^p_n)$.
\end{corollary}

\section{Stability of the number of Alberti representations under Lipschitz perturbations}\label{S:stability}
The purpose of this section is to prove Theorem \ref{T:intro-preserve}. This is allows us to obtain the converse to Corollary \ref{C:main-euclidean-measure} whenever the ambient space is Euclidean, namely Corollary \ref{C:converse}. Unfortunately, we are not able to do this in the setting of general metric spaces. The issue that arises is explained below in Subsection \ref{SS:perturbations}. We obtain results for Euclidean spaces and also some results in the similar spirit for strictly convex finite dimensional Banach spaces.

\subsection{On perturbations of Lipschitz curve fragments}\label{SS:perturbations}

In this subsection we assert a particular property of (small) perturbations of Lipschitz curve fragments. Roughly speaking, the property we require is that the perturbed curve goes in almost the same direction as the original curve for ``most of the time'' provided the perturbation is small enough. 
The problem is that this particular property is not satisfied in general metric spaces. Geometrically, the crucial point is the following. Suppose $\gamma$ is a unit speed curve in some normed linear space $X$ with a uniformly convex unit ball. Suppose $\xi$ is a $1$-Lipschitz curve, which is uniformly very close to $\gamma$ (and having the same domain). Then on a large portion of the domain, $\gamma'-\xi'$ is small. This fails without the convexity assumption. For example, in the plane $\reals^2_\infty$ equipped with the supremum norm, we can consider the pair of curves $\gamma$ and $\xi$ as in Figure \ref{fig:sawtooth}.

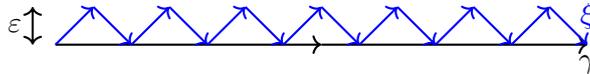
\begin{figure}[h]
\bigskip
    \centering
    \begin{tikzpicture}
    \draw[black, thick, ->] (0, 0) -- (3.5, 0);  
    \draw[black, thick, ->] (3.5, 0) -- (7, 0) node[below] {\(\gamma\)}; 

    \draw[blue, thick, ->] (0, 0) -- ++(0.5, 0.5);
    \draw[blue, thick, ->] (0.5, 0.5) -- ++(0.5, -0.5);
    \draw[blue, thick, ->] (1.0, 0) -- ++(0.5, 0.5);
    \draw[blue, thick, ->] (1.5, 0.5) -- ++(0.5, -0.5);
    \draw[blue, thick, ->] (2.0, 0) -- ++(0.5, 0.5);
    \draw[blue, thick, ->] (2.5, 0.5) -- ++(0.5, -0.5);
    \draw[blue, thick, ->] (3.0, 0) -- ++(0.5, 0.5);
    \draw[blue, thick, ->] (3.5, 0.5) -- ++(0.5, -0.5);
    \draw[blue, thick, ->] (4.0, 0) -- ++(0.5, 0.5);
    \draw[blue, thick, ->] (4.5, 0.5) -- ++(0.5, -0.5);
    \draw[blue, thick, ->] (5.0, 0) -- ++(0.5, 0.5);
    \draw[blue, thick, ->] (5.5, 0.5) -- ++(0.5, -0.5);
    \draw[blue, thick, ->] (6.0, 0) -- ++(0.5, 0.5);
    \draw[blue, thick, ->] (6.5, 0.5) -- ++(0.5, -0.5)
        node[above, yshift=2pt, blue] {\(\xi\)};
    \draw[<->, thick] (-0.3, 0) -- (-0.3, 0.5) 
        node[midway, left] {\(\varepsilon\)};
\end{tikzpicture}
    \caption{The curve $\gamma$ goes in the direction of $e_1$ in the plane $\reals^2_\infty$ equipped with supremum norm. 
    The curve $\xi$ is very close to $\gamma$ uniformly, but the \emph{directions} of $\gamma$ and $\xi$ disagree by a lot on their entire domain. 
    By refining the ``frequency'' of $\xi$, one may achieve $\lVert \gamma-\xi\rVert_\infty<\varepsilon$ for arbitrary $\varepsilon>0$, but one always has $\gamma'=e_1$ and $\xi'=e_1\pm e_2$ a.e.} 
    \label{fig:sawtooth}
\end{figure}

The reason this problem is ``allowed'' to happen is that the curve $\xi$ in fact has length equal to the length of $\gamma$ (note that the projection from the image of $\xi$ to the image of $\gamma$ is an isometry). Such a problem cannot happen in normed spaces whose unit ball is strictly convex.

\begin{lemma}\label{L:rect-estimate}
    Let $(X, |\cdot|)$ be a finite dimensional Banach space.
    Let $I=[a,b]\subset \reals$ be a non-degenerate interval and suppose $\varepsilon>0$. Suppose $\gamma\colon I \to X$ is a Lipschitz curve with unit speed and let
    \begin{equation*}
        u=\frac{\gamma(b)-\gamma(a)}{b-a}.
    \end{equation*}
    Let $f\colon \gamma \to X$ be $1$-Lipschitz with $\lVert f -\Id \rVert_{\infty}\leq \varepsilon$. Suppose $\varphi$ is a linear functional on $\reals^n$ of norm $1$ satisfying also $\varphi(u)=|u|$.
    Let
    \begin{equation*}
        A=1-|u|+\frac{2\varepsilon}{b-a}.
    \end{equation*}
    Then
    \begin{equation*}
        \frac{1}{b-a}\H^1(\{t\in I: \varphi((f\circ \gamma)'(t))> 1-\sqrt{A}\})\geq 1-\sqrt{A}.
    \end{equation*}
\end{lemma}
\begin{proof}
    By the fundamental theorem of calculus, and using also the estimate on the uniform distance of $f$ to the identity, we obtain
    \begin{equation*}
        \int_I (f\circ \gamma)'(t)\diff \H^1(t)\in B((b-a)u, 2\varepsilon).
    \end{equation*}
    Thus, applying the linear map $\varphi$ to both sides, using also linearity of integration and the fact that $\varphi(u)=|u|$, we arrive at
    \begin{equation*}
        \int_I \varphi((f\circ \gamma)'(t))\diff \H^1(t)\geq (b-a)|u|-2\varepsilon.
    \end{equation*}
    Thus, denoting $\xi(t)= \varphi((f\circ \gamma)'(t))$ and dividing both sides by $b-a$, we obtain
    \begin{equation*}
        \frac{1}{b-a}\int_I \xi \diff \H^1 \geq |u|-\frac{2\varepsilon}{b-a}=1-A.
    \end{equation*}
    Since $\varphi$ is of norm $1$, $\xi\leq 1$ $\H^1$-a.e. Therefore the required inequality follows from Lemma \ref{L:av-new}.
\end{proof}

\begin{definition}
    Let $(X, |\cdot|)$ be a normed linear space.
    Given a non-degenerate interval $I\subset \reals$ and a Lipschitz curve $\gamma\colon I \to X$ of unit speed, we call the number
    \begin{equation*}
        \left|\frac{\gamma(b)-\gamma(a)}{b-a}\right|
    \end{equation*}
    the \emph{slope} of $\gamma$.
\end{definition}
We only define slope for curves of unit speed. Moreover if the slope of $\gamma$ is equal to $1$, it is necessary that $\gamma$ is a geodesic. Thus, if $X$ has unique geodesics and $\gamma$ has slope $1$, then $\gamma$ must be a line segment. 

\begin{lemma}\label{L:part}
    Let $(X, |\cdot|)$ be a finite dimensional Banach space.
    Let $I=[a,b]\subset \reals$ be a non-degenerate interval. Suppose $\gamma\colon I \to X$ is a $C^1$ curve with unit speed and let $s<1$, $\varepsilon>0$. Then there exists a partition $a=t_0<t_1<\dots<t_k=b$ of $I$ such that for every $i=1, \dots, n$, the slope of $\gamma_{|[t_{i-1},t_i]}$ is greater than $s$ and, moreover
    \begin{equation*}
        |\gamma'(t)-\frac{\gamma(t_i)-\gamma(t_{i-1})}{t_i-t_{i-1}}|<\varepsilon\quad\text{for every $t\in[t_{i-1}, t_i]$}
    \end{equation*}
\end{lemma}
\begin{proof}
    Since the function $\gamma'\colon I \to \reals^n$ is continuous and $I$ is compact, $\gamma'$ is uniformly continuous. It follows that for any $\varepsilon>0$, there exists a partition $a=t_0<t_1<\dots<t_k=b$ of $I$ such that
    \begin{equation*}
        \sup_{t, s\in [t_{i-1}, t_i]} |\gamma'(t)-\gamma'(s)|\leq \varepsilon \quad \text{for each $i=1, \dots, n$.}
    \end{equation*}
    Let us denote
    \begin{equation*}
        S_i=\frac{1}{t_i-t_{i-1}}\int_{t_{i-1}}^{t_i}\gamma'(t)\diff \H^1(t).
    \end{equation*}
    Then, by the fundamental theorem of calculus, $|S_i|$ is the slope of $\gamma_{|[t_{i-1},t_i]}$.
    We can estimate
    \begin{equation*}
        |S_i-\gamma'(t_{i})|=\left|\frac{1}{t_i-t_{i-1}}\int_{t_{i-1}}^{t_i}\gamma'(t)-\gamma'(t_i)\diff \H^1(t)\right|
        \leq \frac{1}{t_i-t_{i-1}}\int_{t_{i-1}}^{t_i}|\gamma'(t)-\gamma'(t_i)|\diff \H^1(t)\leq \varepsilon.
    \end{equation*}
    Since $\gamma$ has unit speed, it follows that $|\gamma'(t_{i})|=1$ and therefore, by the triangle inequality, $|S_i|\geq 1-\varepsilon$. Thus, the statement follows by possibly reducing $\varepsilon$ so that $\varepsilon\leq 1-s$.
\end{proof}

Recall that a normed linear space $(X, \lVert \cdot \rVert)$, its unit ball denoted by $B_X$, is called \emph{uniformly convex} if for every $\varepsilon>0$, there exists some $\delta>0$ such that whenever $x,y\in B_X$ satisfy $\lVert x- y \rVert\geq \varepsilon$, then
\begin{equation*}
    \left\lVert\frac{x+y}{2}\right\rVert\leq 1-\delta.
\end{equation*}
The space $X$ is called \emph{strictly convex} if the unit sphere of $X$ contains no line segments. If the space $X$ is not strictly convex, then by taking for $x$ and $y$ any pair of distinct points on any line segment in the unit sphere, it follows that $X$ is also not uniformly convex. In other words, uniformly convex spaces are strictly convex.

The following lemma has a standard proof relying on compactness of the unit ball in finite dimensional spaces.
\begin{lemma}\label{L:strict-to-uniform}
    A finite dimensional normed linear space $X$ is strictly convex if and only if it is uniformly convex.
\end{lemma}

\begin{proposition}\label{P:uniform-extremality}
    Suppose a normed linear space $X$ is uniformly convex. Then to each $u$ in $\partial B_X$, the unit sphere of $X$, a unit linear functional $\varphi_u\in X^*$ can be found in such a way that the following holds. To each $\varepsilon>0$ there exists $\delta>0$ such that whenever $u\in \partial B_X$, $v\in B_X$ and $\varphi_u(u-v)< \delta$, then $\lVert u - v \rVert < \varepsilon$.
\end{proposition}
\begin{proof}
    By Hahn-Banach, for each $u\in\partial B_X$ there exists a $\varphi_u\in X^*$ such that $\varphi_u(u)=1$ and $\lVert \varphi \rVert = 1$.
    Let $\varepsilon>0$ and suppose $\delta>0$ is from the definition of uniform convexity. That is, for any $u,v\in B_X$ with $\lVert u- v \rVert\geq \varepsilon$, it holds that
\begin{equation}\label{E:unif-conv}
    \left\lVert\frac{u+v}{2}\right\rVert\leq 1-\delta.
\end{equation}
    Now let $u\in\partial B_X$ and $v\in B_X$ be such that
    \begin{equation}\label{E:varphi-u-est1}
        \varphi_u(u-v)<\delta.
    \end{equation}
    For contradiction, let us also assume that $\lVert u -v \rVert \geq \varepsilon$. We may then use \eqref{E:unif-conv} together with $\lVert \varphi \rVert = 1$ to assert that
    \begin{equation}\label{E:varphi-u-est2}
        \varphi_u(u) + \varphi_u(v)\leq 2- 2\delta.
    \end{equation}
    Adding the inequalities \eqref{E:varphi-u-est1} and \eqref{E:varphi-u-est2} yields
    \begin{equation*}
        2\varphi_u(u)< 2 - \delta.
    \end{equation*}
    This contradicts the fact that $\varphi_u(u)=1$.
\end{proof} 

It can be immediately verified, that if $X$ has finite dimension and satisfies the property in the conclusion of Proposition \ref{P:uniform-extremality}, then it must be strictly convex, and therefore, by Lemma \ref{L:strict-to-uniform}, also uniformly convex. 

\begin{remark}
    The arguments in the remainder of this section, concerning the tangents of Lipschitz curves and directions of Alberti representations are carried out in the setting of strictly convex, \emph{finite dimensional} Banach spaces $X$. In these spaces, the conclusion of Proposition \ref{P:uniform-extremality} holds, which is of principal importance. 
    Since any uniformly convex Banach space $X$ possesses the Radon-Nikod\'ym property, an appropriate notion of a tangent to a Lipschitz curve $\gamma\colon [0,1] \to X$ can be defined (and exist $\H^1$-a.e.). Moreover, Proposition \ref{P:uniform-extremality} does \emph{not} assume finite dimension of $X$. Therefore, it would also be possible to work in infinite dimensional uniformly convex Banach spaces. This extra generality, however, is not of so much importance and would require significant additional preliminaries, which is why we only work in the finite dimensional setting.
\end{remark}

\begin{theorem}\label{T:con-in-measure-C1}
    Suppose $(X, |\cdot|)$ is a finite dimensional strictly convex Banach space. Let $I=[a,b]\subset \reals$ be a non-degenerate interval and $\gamma\colon I \to X$ a $C^1$ curve of unit speed. Then, to each $\delta>0$, there exists $\varepsilon>0$ such that the following holds. If $f\colon \gamma \to X$ is $1$-Lipschitz and $\lVert f- \Id \rVert_\infty<\varepsilon$, then
    \begin{equation}\label{E:cont-in-measure-for-C^1}
        \frac{1}{b-a}\H^1(\{t\in I: |(f\circ \gamma)'(t)-\gamma'(t)|\geq \delta\})\leq \delta.
    \end{equation}
\end{theorem}
\begin{proof}
    Let us denote by $B$ the unit ball of $X$.
    For $u\in \partial B$ let $\varphi_u$ be from Proposition \ref{P:uniform-extremality}. There exists some $\delta\geq A>0$ such that for any $u\in \partial B$ and any $v\in B$
    \begin{equation*}
        \varphi_u(u-v)\leq \sqrt{A} \implies |u-v|\leq \frac{\delta}{3}.
    \end{equation*}
    Note that in particular, since $\varphi_u(u)=1$, we have
    \begin{equation}\label{E:est-via-proj}
        \varphi_u(v)\geq 1- \sqrt{A} \implies |u-v|\leq \frac{\delta}{3}.
    \end{equation}
    Now let $s<1$ be such that $s>1-A$ and $1-s\leq \frac{\delta}{3}$. For any $\varepsilon>0$, using Lemma \ref{L:part}, we can find a partition $a=t_0<\dots, <t_k=b$ of $I$ such that the slope of each $\gamma_{|[t_{i-1}, t_i]}$ is at least $s$ and, moreover, if we denote
    \begin{equation*}
        w_i=\frac{\gamma(t_i)-\gamma(t_{i-1})}{t_i-t_{i-1}};\quad A_i=1-|w_i|+ \frac{2\varepsilon}{t_i-t_{i-1}},
    \end{equation*}
    then
    \begin{equation}\label{E:slope-to-curve-est}
        |w_i-\gamma'(t)|\leq \frac{\delta}{3}\quad \text{for all $t\in[t_{i-1}, t_i]$.}
    \end{equation}
    Note that the condition on slopes is equivalent to $|w_i|\geq s$.
    Let
    \begin{equation*}
        r=\min_{i=1, \dots, k} t_{i}-t_{i-1}.
    \end{equation*}
    Finally, we choose the $\varepsilon>0$ which is defined by
    \begin{equation*}
        A=1-s+\frac{2\varepsilon}{r}.
    \end{equation*}
    Then $A_i\leq A$.
    Let $u_i=\frac{w_i}{|w_i|}$ and $\varphi_i=\varphi_{u_i}$. Then $\varphi_i(w_i)=|w_i|$ and thus we may use Lemma \ref{L:rect-estimate} to obtain
    \begin{equation*}
        \frac{1}{t_i-t_{i-1}}\H^1(\{t\in[t_{i-1}, t_i]: \varphi_i((f\circ\gamma)'(t))>1-\sqrt{A_i}\})\geq 1-\sqrt{A_i}.
    \end{equation*}
    Using \eqref{E:est-via-proj}, $A_i\leq A$ and $\sqrt{A}\leq \delta$, this implies
    \begin{equation*}
        \frac{1}{t_i-t_{i-1}}\H^1(\{t\in[t_{i-1}, t_i]: |(f\circ\gamma)'(t)-u_i|\leq \frac{\delta}{3}\})\geq 1-\sqrt{A_i}\geq 1-\sqrt{A}\geq 1-\delta.
    \end{equation*}
    Thus, passing to the complement,
    \begin{equation*}
        \frac{1}{t_i-t_{i-1}}\H^1(\{t\in[t_{i-1}, t_i]: |(f\circ\gamma)'(t)-u_i|> \frac{\delta}{3}\})\leq \delta.
    \end{equation*}
    Now let $t\in [t_{i-1}, t_i]$ be such that $|(f\circ\gamma)'(t)-u_i|> \frac{\delta}{3}$. Then using the estimate \eqref{E:slope-to-curve-est}, $1-s< \frac{\delta}{3}$ and the triangle inequality yields
    \begin{equation*}
        |(f\circ\gamma)'(t)-\gamma'(t)|\leq|(f\circ\gamma)'(t)-u_i|+|u_i-w_i|+|w_i-\gamma'(t)|\leq \frac{\delta}{3}+|1-s|+\frac{\delta}{3}\leq \delta.
    \end{equation*}
    Thus, combining the last two estimates, we arrive at
    \begin{equation*}
        \frac{1}{t_i-t_{i-1}}\H^1(\{t\in[t_{i-1}, t_i]: |(f\circ\gamma)'(t)-\gamma'(t)|> \delta\})\leq\delta.
    \end{equation*}
    The required inequality \eqref{E:cont-in-measure-for-C^1} follows by summing over $i=1, \dots, k$.
\end{proof}


\begin{lemma}\label{L:conv-hulls-of-neighb}
    Suppose $K$ is a compact subset of a Banach space $X$ and $C\subset X$ is a convex set containing an open neighbourhood of $K$. Then $C$ also contains an open neighbourhood of $\conv K$.
\end{lemma}
\begin{proof}
    Let $G$ be an open neighbourhood of $K$ contained in $C$. Since $K$ is compact, there exists some $\delta>0$ such that $B(K,\delta)\subset G$. In particular, for every $v\in X$ and $x\in K$ with $\lVert v \rVert\leq \delta$, it holds that $x+v\in G$.

    Fix $x\in \conv K$. To show that $C$ contains on open neighbourhood of $\conv K$, it suffices to show that $x+v\in C$ whenever $\lVert v \rVert\leq \delta$. By definition, there are $\lambda_1,\dots, \lambda_N\in[0,1]$ with $\sum_i\lambda_i=1$ and $x_1, \dots, x_N\in K$ such that $x=\sum_i\lambda_i x_i$. Then it holds that
    \begin{equation*}
        x+v=\sum_i\lambda_i(x_i+v).
    \end{equation*}
    It follows that $x+v\in\conv G\subset \conv C = C$.
\end{proof}

\begin{theorem}\label{T:con-in-measure-frag}
    Suppose $(X, |\cdot|)$ is a finite dimensional strictly convex Banach space and let $\gamma \in \Gamma(X)$ with $\H^1(\dom\gamma)>0$. Suppose $C\subset X$ is a convex set containing an open neighbourhood of $\gamma$. Then, to each $\delta>0$, there exists $\varepsilon>0$ such that the following holds. If $f\colon C \to X$ is $1$-Lipschitz and $\lVert f- \Id \rVert_\infty<\varepsilon$, then
    \begin{equation*}
        \frac{1}{\H^1(\dom\gamma)}\H^1(\{t\in \dom\gamma: |(f\circ \gamma)'(t)-\gamma'(t)|\geq \delta\})\leq \delta.
    \end{equation*}
\end{theorem}
\begin{proof}
    This follows from Theorem \ref{T:con-in-measure-C1} via a standard measure-theoretic argument, using the fact that Lipschitz curve fragments agree, on an arbitrarily large subset of the domain, with restrictions of $C^1$ curves. Therefore, we only give a proof sketch.

    The complement of $\dom\gamma$ in $\conv(\dom\gamma)$ is a countable union of disjoint open intervals (the collection of the intervals is unique). One may extend $\gamma$ in an affine way (uniquely) on each of those open intervals to obtain a Lipschitz curve $\overline{\gamma}\colon\conv(\dom\gamma)\to \reals^n$.
    Since $C$ is convex, it contains, via Lemma \ref{L:conv-hulls-of-neighb}, an open neighbourhood of the image of $\overline{\gamma}$. Let us denote $I=\conv(\dom\gamma)$.

    By Whitney's extension theorem, for any $\eta>0$, $\delta>0$  there exists a $C^1$-curve $\beta$ (with the domain $I$) such that $\beta$ agrees with $\overline{\gamma}$ on a set $E\subset I$ with $\H^1(I\setminus E)<\eta$
    and, moreover, $\lVert \overline{\gamma}-\beta\rVert_\infty< \delta$.
    Thus, by choosing $\delta$ small enough, we may obtain that $\beta$ lies inside $C$.
    By Lebesgue's density theorem, $\beta'$ agrees with $\overline{\gamma}'$ a.e.~in $E$. Thus, since $\eta$ may be taken arbitrarily small, we see that it is sufficient to prove the statement for $C^1$ curves.

    Thus, let us assume, without loss of generality, that $\gamma$ is a $C^1$ curve. Now we may find an interval $J$ and a $C^1$ curve of unit speed $\beta$ such that the image of $\beta$ agrees with the image of $\gamma$ and, whenever $\gamma'(t)\not=0$, we have $\frac{\gamma'(t)}{|\gamma'(t)|}=\beta'(\beta^{-1}(\gamma(t)))$. If $\gamma'(t)=0$, we also have $(f\circ\gamma)'(t)=0$, and if not, then
    \begin{equation*}
        (f\circ \beta)(\beta^{-1}(\gamma(t)))=(f\circ\gamma)'(t)\frac{1}{|\gamma'(t)|}.
    \end{equation*}
    Thus, it is enough to show the statement for $C^1$ curves of unit speed. This is exactly what Theorem \ref{T:con-in-measure-C1} asserts.
\end{proof}

\subsection{Restrictions and pushforwards of Alberti representations}

In this subsection we fix a strictly convex Banach space $X$ and a sequence of 1-Lipschitz maps $f_i\colon D \to X$ defined on some subset of $X$ that converge uniformly to the identity.
In Corollary \ref{C:pw-main}, we will show that, if a measure $\mu$ on $D$ has $d$ independent Alberti representations (for some $d\in\nat$, defined in Definition \ref{D:AR}), then so does a large proportion of $(f_i)_{\#}\mu$, for sufficiently large $i$.

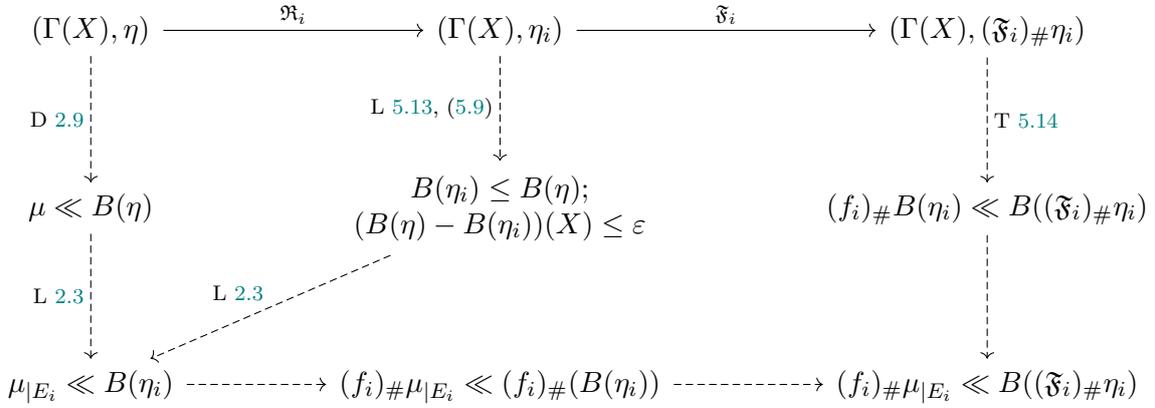
\begin{figure}[ht]
    \centering
\[\begin{tikzcd}
	{(\Gamma(X), \eta)} && {(\Gamma(X), \eta_i)} && {(\Gamma(X), (\mathfrak{F}_i)_\#\eta_i)} \\
	\\
	{\mu\ll B(\eta)} && \begin{array}{c} B(\eta_i)\leq B(\eta);\\ (B(\eta)-B(\eta_i))(X)\leq\varepsilon \end{array} && {(f_i)_\#B(\eta_i)\ll B((\mathfrak{F}_i)_\#\eta_i)} \\
	\\
	{\mu_{|E_i}\ll B(\eta_i)} && {(f_i)_\#\mu_{|E_i}\ll(f_i)_\#(B(\eta_i))} && {(f_i)_\#\mu_{|E_i}\ll B((\mathfrak{F}_i)_\#\eta_i)}
	\arrow["{\mathfrak{R}_i}", from=1-1, to=1-3]
	\arrow["{\text{D \ref{D:AR}}}"', dashed, from=1-1, to=3-1]
	\arrow["{\mathfrak{F}_i}", from=1-3, to=1-5]
	\arrow["\text{L \ref{L:D-convergence}, \eqref{E:barycenter-estimate}}"', dashed, from=1-3, to=3-3]
	\arrow["\text{T \ref{T:commutativity}}", dashed, from=1-5, to=3-5]
	\arrow["\text{L \ref{L:ac-with-error}}"', dashed, from=3-1, to=5-1]
	\arrow["\text{L \ref{L:ac-with-error}}"', dashed, from=3-3, to=5-1]
	\arrow["", dashed, from=3-5, to=5-5]
	\arrow[""', dashed, from=5-1, to=5-3]
	\arrow[""', dashed, from=5-3, to=5-5]
\end{tikzcd}\]
    \caption{Diagram representing the argument used to obtain Alberti representation of pushforwards.}
    \label{fig:diagram}
\end{figure}

The proof of this statement uses a lot of notation and fairly abstract definitions since fairly non-trivial measurability issues need to be resolved. Yet, given the theory on restriction operators developed developed in Subsection \ref{sec:restriction}, our statements follow essentially by unpacking the relevant definitions and using the preliminary results. Figure \ref{fig:diagram} illustrates the argument. The full arrows therein correspond to operations, while the dotted arrows don't have a specific meaning, they simply indicate, informally,  the direction of the argument we are carrying out. Our goal is the bottom right node in the diagram. We remark that, where two arrows enter a node, they are both requirements for the argument, not alternate approaches.

Let $X$ be a strictly convex finite dimensional Banach space. Let $\mu$ be a finite Borel measure on $X$, which is moreover supported inside a compact set $K\subset X$. Suppose further that $\eta$ is an Alberti representation of $\mu$. Let $C\subset X$ be convex and assume that $C$ contains a neighbourhood of $K$.

Let us fix $u\in X^*$ and a thickness $\theta\in(0,1)$ and recall the definition of the cone
\begin{equation*}
    C(u,\theta)=\{v\in X: \langle u, v \rangle \geq \lVert v \rVert_X (1-\theta)\}.
\end{equation*}

Assume that there is a Borel set $\mathcal{U}\subset \Gamma(X)$ of full $\eta$-measure and there are $m>0$ and $L<\infty$ such that every $\gamma\in \mathcal{U}$ has the following properties:
\begin{enumerate}
    \item \label{Enum:speed} for $\H^1$-a.e.~$t\in\dom\gamma$, $|\gamma'|(t)\geq \delta$,
    \item \label{Enum:cone} for $\H^1$-a.e.~$t\in\dom\gamma$, $\gamma'(t)\in C(u,\theta)$,
    \item \label{Enum:domain} $\H^1(\dom\gamma)\geq m$,
    \item \label{Enum:Lipschitz} every $\gamma\in \mathcal{U}$ is $L$-Lipschitz.
\end{enumerate}
Let $X\supset D \supset C$ and suppose $f_i\colon D \to X$ is a sequence of $1$-Lipschitz maps converging uniformly to the identity. Given any Lipschitz map $g\colon X \to X$, we define $\E^u_g\colon P(X)\to \reals$ by
\begin{equation*}
    \E_g^u(\gamma,t)=\frac{1}{\lVert (g\circ \gamma)'(t)\rVert_X}\langle (g\circ \gamma)', u\rangle,
\end{equation*}
whenever the right-hand side is well defined, and we let $\E_g^u=0$ otherwise. Similarly, we define $\mathcal{F}_g\colon P(X)\to \reals$ by
\begin{equation*}
    \mathcal{F}_g(\gamma, t)=\lVert(g\circ \gamma)'(t)\rVert_X,
\end{equation*}
whenever the right-hand side is well defined, and we let $\mathcal{F}_g=0$ otherwise. The following lemma is an immediate consequence of the definitions.

\begin{lemma}\label{L:cone-vs-speed}
    For a given curve fragment $\gamma$, it holds that $\E^u_g(\gamma, t)\geq 1-\theta$ for $\H^1$-a.e.~$t\in\dom\gamma$ if and only if $g\circ\gamma$ goes in the direction of the cone $C(u,\theta)$. 
\end{lemma}

Finally we fix $\delta>r>0$ and introduce the following notation.
\begin{equation*}
        \mathcal{G}^i_{\textnormal{cone}}=\{(\gamma, t)\in P(X): \E^u_{f_i}(\gamma, t)\geq 1-\theta -r\};\quad
        \mathcal{G}^i_{\textnormal{speed}}=\{(\gamma, t)\in P(X):\F_{f_i}(\gamma, t)\geq \delta- r\}.
\end{equation*}
Since the metric space $P(X)$ (defined at the beginning of Subsection \ref{sec:alberti}) is complete and separable and the measure $(\eta\times\H^1_{[0,1]})_{P(X)}$ is Borel, it is also inner regular by compact sets. Thus, for each $i\in\nat$ there are compact sets $\mathcal{K}^i_{\textnormal{cone}}\subset \mathcal{G}^i_{\textnormal{cone}}$ and $\mathcal{K}^i_{\textnormal{speed}}\subset\mathcal{G}^i_{\textnormal{speed}}$ such that
\begin{equation}\label{E:size-cone-est}
    (\eta\times\H^1_{[0,1]})_{P(X)}(\mathcal{G}^i_{\textnormal{cone}}\setminus \mathcal{K}^i_{\textnormal{cone}})\leq \frac{1}{i} (\eta\times\H^1_{[0,1]})_{P(X)}(\mathcal{G}^i_{\textnormal{cone}})
\end{equation}
and
\begin{equation}\label{E:size-speed-est}
    (\eta\times\H^1_{[0,1]})_{P(X)}(\mathcal{G}^i_{\textnormal{speed}}\setminus \mathcal{K}^i_{\textnormal{speed}})\leq \frac{1}{i} (\eta\times\H^1_{[0,1]})_{P(X)}(\mathcal{G}^i_{\textnormal{speed}}).
\end{equation}
We let $\mathcal{K}^i=\mathcal{K}^i_{\textnormal{cone}}\cap \mathcal{K}^i_{\textnormal{speed}}$.

\begin{theorem}\label{T:m-convergence-for-speeds}
    It holds that
    \begin{equation*}
        \E_{f_i}^u\overset{m}{\to}\E^u_{\Id}\quad \text{and}\quad \F_{f_i}\overset{m}{\to}\F_{\Id} \quad \text{on $K$.}
    \end{equation*}
\end{theorem}
\begin{proof}
    Both of the convergences are an immediate consequence of Theorem \ref{T:con-in-measure-frag}.
\end{proof}

\begin{corollary}\label{C:Ki-fill}
    It holds that
    \begin{equation*}
        (\eta\times\H^1_{[0,1]})_{P(X)}(P(X)\setminus \mathcal{K}^i) \to 0.
    \end{equation*}
\end{corollary}
\begin{proof}
    This is an immediate consequence of Theorem \ref{T:m-convergence-for-speeds}, Lemma \ref{L:m-to-measure}, Lemma \ref{L:cone-vs-speed} together with \ref{Enum:cone} and the estimates \eqref{E:size-cone-est} and \eqref{E:size-speed-est}.
\end{proof}

Recall the projection map $p\colon P(X)\to \Gamma(X)$ given by $p(\gamma,t)=\gamma$.
For each $i\in\nat$ we define
\begin{equation*}
    \mathcal{U}_i=p(\mathcal{K}_i)\cap \mathcal{U}
\end{equation*}
and the restriction $\mathfrak{R}_i\colon \mathcal{U}_i\to \Gamma(X)$ given by
\begin{equation*}
    \mathfrak{R}_i(\gamma)=\gamma_{|\{t\in\dom\gamma: (\gamma,t)\in \mathcal{K}^i\}}.
\end{equation*}
The restriction $\mathfrak{R}_i$ is a Borel restriction on $\mathcal{U}_i$ by Theorem \ref{T:Borel-restriction}.
Finally, we let $\mathcal{D}_i\colon \mathcal{U} \to [0,1]$ be equal to the relative density of $\mathfrak{R}_i$ on $\mathcal{U}_i\subset \mathcal{U}$ and equal to $0$ otherwise. We let
\begin{equation*}
    \eta_i= (\mathfrak{R}_i)_\#(\eta_{|\mathcal{U}_i}).
\end{equation*}
Because $\eta=\eta_{|\mathcal{U}_i}+\eta_{|\mathcal{U}\setminus\mathcal{U}_i}$ and since the barycenter is linear, we have
\begin{equation*}
    B(\eta)= B(\eta_{|\mathcal{U}_i})+B(\eta_{|\mathcal{U}\setminus\mathcal{U}_i}).
\end{equation*}
Thus, using the estimate (recall $\mathcal{D}_i$ is set to $0$ outside $\mathcal{U}_i$)
\begin{equation*}
    B(\eta_{|\mathcal{U}\setminus\mathcal{U}_i})(X)\leq L\eta(\mathcal{U}\setminus\mathcal{U}_i)
    =\int_{\mathcal{U}\setminus\mathcal{U}_i}(1-\mathcal{D}_i(\gamma))L\diff \eta,
\end{equation*}
in combination with Theorem \ref{T:density-est-for-restrictions} we obtain
\begin{equation}\label{E:barycenter-estimate}
    B( \eta_i)\leq B(\eta)\quad\text{and}\quad (B(\eta)-B(\eta_i))(X)\leq \int_{\mathcal{U}}(1-\mathcal{D}_i(\gamma))L\diff\eta(\gamma).
\end{equation}

\begin{lemma}\label{L:D-convergence}
    It holds that $(B(\eta)-B(\eta_i))(X) \to 0$.
\end{lemma}
\begin{proof}
    Let $i_j$ be an increasing sequence of indices and consider the sequence $(B(\eta)-B(\eta_{i_j}))(X)$.
    Using \ref{Enum:domain}, we see that $p_\#(\eta \times \H^1_{[0,1]})\geq \frac{1}{m} \eta$ and therefore Corollary \ref{C:Ki-fill} implies $\eta(\mathcal{U}\setminus\mathcal{U}_{i_j})\to 0$.
    Thus, there exists a further increasing sequence of indices $j_k$ such that for $\eta$-a.e. $\gamma\in \mathcal{U}$, there exists some $k_0\in\nat$ such that for all $k\geq k_0$, $\gamma\in \mathcal{U}_{{i_j}_k}$.
   From Fubini's theorem, it follows that for $\eta$-a.e.~such $\gamma\in \mathcal{U}$, we have, for the slices of $\mathcal{K}^i$,
    \begin{equation*}
        \H^1(\{t\in[0,1]: (\gamma,t)\in \mathcal{K}^{{i_j}_k}\})\to \H^1(\{t\in[0,1]:(\gamma,t)\in P(X)\}) = \H^1(\dom\gamma).
    \end{equation*}
    Thus, $\mathcal{D}_{{i_j}_k}\to 1$ $\eta$-a.e. The Lebesgue dominated convergence theorem together with \eqref{E:barycenter-estimate} implies $(B(\eta)-B(\eta_{{i_j}_k}))(X)\to 0$. Since the sequence $i_j$ was arbitrary, it follows that $(B(\eta)-B(\eta_i))(X) \to 0$.
\end{proof}

Let us now define the maps $\mathfrak{F}_i\colon \Gamma(X)\to \Gamma(X)$ given by $\mathfrak{F}_i(\gamma)=(f_i\circ \gamma)$. Since $f_i$ are Lipschitz, so are $\mathfrak{F}_i$, in particular, they are also Borel. 

\begin{theorem}[Commutatitvity of pushforwards and barycenters]\label{T:commutativity}
    For any $\gamma\in\mathfrak{R}_i(\mathcal{U}_i)$, the curve fragment $(f_i\circ \gamma)$ goes in the direction of $C(u,\theta+r)$ and has speed at least $\delta-r$. Moreover,
    \begin{equation*}
        (f_i)_\# B(\eta_i)\ll B((\mathfrak{F}_i)_\#\eta_i).
    \end{equation*}
\end{theorem}
\begin{proof}
    By definition of $\mathfrak{R}_i$, each $\gamma\in \mathfrak{R}_i(\mathcal{U}_i)$ satisfies $\gamma\subset K$ with the additional property that $(f_i\circ\gamma)$ goes in the direction of the cone $C(u, \theta +r)$ and has speed at least $\delta-r$. Indeed, the speed follows immediately from the definition of $\eta_i$ and the cone follows from the definition and Lemma \ref{L:cone-vs-speed}.

    For the second statement, let us assume that $E\subset X$ is a Borel set such that
    \begin{equation*}
        0= B((\mathfrak{F}_i)_\#\eta_i)=\int_{\Gamma(X)}\int_\gamma \chi_E \diff\H^1\diff [(\mathfrak{F}_i\circ \mathfrak{R_i})_\#(\eta_{|\mathcal{U}_i})](\gamma)
        =\int_{\mathcal{U}_i}\int_{f_i\circ \gamma_{|\{t:(\gamma,t)\in\mathcal{K}^i\}}}\chi_E\diff \H^1 \diff \eta.
    \end{equation*}
    For a curve $\gamma\in \Gamma(X)$, let us denote $A_\gamma=\{t\in\dom\gamma: (\gamma, t)\in \mathcal{K}^i\;\text{and}\; f_i(\gamma(t))\in E\}$.
    Using the equality above, together with the fact that $\eta$-a.e.~$\gamma$ is injective with some minimal speed, we obtain for $\eta$-a.e.~$\gamma$
    \begin{equation}\label{E:measurability-0}
        0=\H^1((f_i\circ \gamma_{|\{t:(\gamma,t)\in\mathcal{K}^i\}})^{-1}(E))=\H^1(A_\gamma).
    \end{equation}
    By unpacking the definition, one can verify that
    \begin{equation*}
        (f_i)_\#B(\eta_i)(E)=\int_{\Gamma(X)}\int_{A_\gamma}|\gamma'(t)|\diff \H^1(t) \diff \eta,
    \end{equation*}
    which equals zero by \eqref{E:measurability-0}.   
\end{proof}

\begin{theorem}\label{T:pw-main}
    For any $\varepsilon>0$, there exists some $i_0\in \nat$ such that for every $i\geq i_0$ a Borel set $E_i\subset K$ can be found such that
    \begin{equation*}
        \mu(K\setminus E_i)\leq \varepsilon
    \end{equation*}
     and $(\mathfrak{F}_i)_\#\eta_i$ is an Alberti representation of $(f_i)_\#\mu_{|E_i}$.
\end{theorem}
\begin{proof}
    By Theorem \ref{T:commutativity}, $(\mathfrak{F}_i)_\#\eta_i$ is supported on injective curves. The rest follows by application of Lemma \ref{L:ac-with-error}, Theorem \ref{T:commutativity} and the estimate \eqref{E:barycenter-estimate} together with Lemma \ref{L:D-convergence}.
\end{proof}

Now it is time to consider a number of different Alberti representations. We fix functionals $u^1, \dots, u^d\in X^*$ and thicknesses $\theta^1, \dots, \theta^d$ and Alberti representations $\eta^1, \dots, \eta^d$ of the measure $\mu$. We assume each $\eta^j$ is supported in a set $\mathcal{U}^j$, having the properties \ref{Enum:pw-constr-supp}, $\dots$, \ref{Enum:pw-constr-length} above, where we replace $C(u,\theta)$ with $C(u^j, \theta^j)$. We assume that the cones $C(u^j, \theta^j)$ are independent, so that there is some $r>0$ with $\delta> r$ and such that the cones $C(u^j, \theta^j+ r)$ are still independent.

\begin{corollary}\label{C:pw-main}
    For any $\varepsilon>0$, there exists some $i_0\in\nat$ such that for every $i\geq i_0$ a Borel set $E_i\subset K$ can be found such that
    \begin{equation*}
        \mu(K\setminus E_i)\leq \varepsilon
    \end{equation*}
    and $(f_i)_\# \mu_{|E_i}$ has $d$ independent Alberti representations.
\end{corollary}
\begin{proof}
    Let $\varepsilon_1, \dots, \varepsilon_d>0$ be such that $\sum_j \varepsilon_j\leq \varepsilon$. For each $\varepsilon_j$, apply Theorem \ref{T:pw-main} so as to obtain $i_j$ and for every $i\geq i_j$ a Borel set $E^j_i\subset K$ with
    \begin{equation*}
        \mu(K\setminus E^j_i)\leq \varepsilon_j
    \end{equation*}
    and such that $(\mathfrak{F}_i)_\#\eta_i^j$ is an Alberti representation of $\mu_{|E^j_i}$. Let $i_0=\max_{j=1,\dots, d}\{i_j\}$ and fix $i\geq i_0$. By Theorem \ref{T:commutativity}, $(\mathfrak{F}_i)_\#\eta_i^j$-a.e.~$\gamma$ goes in the direction of $C(u^j,\theta^j+r)$ and has speed at least $\delta-r$. It follows therefore that $(\mathfrak{F}_i)_\#\eta_i^1, \dots, (\mathfrak{F}_i)_\#\eta_i^d$ are independent Alberti representations of
    \begin{equation*}
        \mu_{|E_i}, \quad \text{where}\quad E_i=\cap_{j=1}^d E^j_i.
    \end{equation*}
    Our condition on the sum of $\varepsilon_j$'s implies that
    \begin{equation*}
        \mu(K\setminus E_i)\leq\mu(\bigcup_{j=1}^d K\setminus E^j_i)\leq \sum_{j=1}^d \varepsilon_j\leq \varepsilon.
    \end{equation*}

\end{proof}

\subsection{Stability of dimension under perturbations}
We now conclude the proof of Theorem \ref{T:intro-preserve}, which is an immediate consequence of the following Theorem \ref{T:stability-for-X}.
\begin{theorem}\label{T:stability-for-X}
    Suppose $X$ is a strictly convex finite dimensional Banach space. Let $D\subset X$ be a Borel set. Let $\mu$ be a $\sigma$-finite Borel measure on $D$. Suppose there is a compact set $K\subset D$ such that $\infty>\mu(K)>0$ and $\dim_T \mu_{|K} \geq d \in \nat$. Suppose also that there is a convex set $C$ containing an open neighbourhood of $K$ such that $C\subset D$. Then any $1$-Lipschitz map $F\colon D \to X$ which satisfies $F_{|C}=\Id_{|C}$ is in the interior of
    \begin{equation*}
        \{f\in \tLip_1(D,X): \dim_T f_\#\mu \geq d\}.
    \end{equation*}
\end{theorem}
\begin{proof}
    The proof is complete once we show that Corollary \ref{C:pw-main} applies.

    Let $f_i\in \tLip_1(D,X)$ converge to any $F$ satisfying $F_{|C}=\Id_{|C}$.
    There exists a compact subset $\tilde{K}\subset K$ such that $\mu_{|\tilde{K}}$ has $d$-independent Alberti representations. It is without loss of generality to assume $\tilde K = K$. Thus, let us take $\eta^1, \dots, \eta^d$ Alberti representations of $\mu$ going in the directions of independent cones $C(u^i,\theta^i)$. It is without loss of generality to assume that each $\eta^j$ is supported on curve fragments $\gamma \subset K$.
    
    Since $K$ is compact and $\eta_j$-a.e.~$\gamma$ goes in the direction of $C(u^i, \theta^i)$, there exists some $M<\infty$ such that $\eta_j$-a.e.~$\gamma$ has length at most $M$. Using a standard exhaustion argument, for every $\varepsilon>0$, there exists a Borel set $\mathcal{U}^j\subset \Gamma(X)$ and $\delta_j, m_j>0$, $L_j<\infty$ such that
    \begin{equation}\label{E:restr-final}
        \eta^j(\Gamma(X)\setminus \mathcal{U}^j)\leq \varepsilon,
    \end{equation}
    and
    \begin{enumerate}
        \item\label{Enum:pw-constr-supp} every $\gamma\in \mathcal{U}^j$ is supported inside $K$,
        \item\label{Enum:pw-constr-cone} every $\gamma\in \mathcal{U}^j$ goes in the direction of $C(u^j, \theta^j)$ and has speed at least $\delta_j$,
        \item\label{Enum:pw-constr-Lip} every $\gamma\in \mathcal{U}^j$ is $L_j$-Lipschitz,
        \item\label{Enum:pw-constr-length} every $\gamma\in \mathcal{U}^j$ satisfies $\H^1(\dom\gamma)\geq m_j$.
    \end{enumerate}
    Let us take $\delta=\min_{j=1,\dots, d}\{\delta_j\}$, $m=\min_{j=1,\dots, d}\{m_j\}$, $L=\max_{j=1,\dots, d}\{L_j\}$. The inequality \eqref{E:restr-final} implies
    \begin{equation*}
        (B(\eta^j)-B(\eta^j_{|\mathcal{U}^j}))(X)\leq M \varepsilon.
    \end{equation*}
    Thus, by taking $\varepsilon>0$ small enough, there exists a set $\tilde K \subset K$ with $\mu(\tilde K)>0$, such that $\eta^j_{|\mathcal{U}^j}$ is an Alberti representation of $\mu_{\tilde K}$ for each $j=1,\dots, d$. 

    Now we may apply Corollary \ref{C:pw-main} to the measure $\mu_{|\tilde K}$ to obtain $i_0\in\nat$ such that for all $i\geq i_0$, $(f_i)_\# \mu_{|\tilde K}$ has $d$-independent Alberti representations. It follows that for such $i$, $\dim_T (f_i)_\# \mu\geq d$. Thus we have shown that for any sequence converging to $F$, eventually the elements of the sequence are in the set $ \{f\in \tLip_1(D,X): \dim_T f_\#\mu \geq d\}$. This is equivalent to $F$ being in the interior of the set.
\end{proof}

\begin{theorem}\label{T:stability-for-Euclidean}
    Let $\mu$ be a Borel measure on $\reals^k$ with $\dim_T \mu\geq d$. Let $m\geq d$. Then the set
    \begin{equation}\label{E:the-set}
        \{f\in \tLip_1(\reals^k, \reals^m): \dim_T f_\# \mu\geq d\}
    \end{equation}
    has non-empty interior. For any compact set $K\subset \reals^k$ such that $\dim_T \mu_{|K}\geq d$, any $F\in \tLip_1(\reals^k, \reals^m)$ satisfying $F_{|C}=I_{|C}$ for some isometry $I\colon \reals^k \to \reals^m$ and a convex set $C$ containing an open neighbourhood of $K$, is in the interior of the set in \eqref{E:the-set}.
\end{theorem}
\begin{proof}
    This follows from the fact that $\reals^k$ is strictly convex and Theorem \ref{T:stability-for-Euclidean}. Indeed, we see that the statement holds true if $k=m$ and $I$ is the identity. Therefore, it also holds if $I$ is any isometry. If $m>k$, let $P_I$ be the orthogonal projection in $\reals^m$ to $I(\reals^k)$. Since $P_I$ is $1$-Lipschitz, we see that for any sequence $f_i\in \tLip_1(\reals^k, \reals^m)$ with $f_i\to F$ uniformly, we also have $P_I \circ f_i \to F$ uniformly on $C$. Hence, there exists some $i_0\in \nat$ such that for all $i\geq i_0$, $\dim_T(P_I\circ f_i)_\#\mu \geq d$. Since $1$-Lipschitz pushforwards cannot increase the dimension $\dim_T$, it follows that $\dim_T (f_i)_\# \mu \geq d$.
\end{proof}

We now collect the relevant corollaries which in particular immediately imply Corollary \ref{C:converse}.
The following results follow from Lemma \ref{L:DeRindler}
and Theorems \ref{T:stability-for-X} and \ref{T:stability-for-Euclidean} respectively.

\begin{corollary}\label{C:stability-for-X-Haus}
     Suppose $X$ is a strictly convex finite dimensional Banach space. Let $D\subset X$ be a Borel set. Let $\mu$ be a $\sigma$-finite Borel measure on $D$. Suppose there is a compact set $K\subset D$ such that $\infty>\mu(K)>0$ and $\dim_T \mu_{|K} \geq d \in \nat$. Suppose also that there is a convex set $C$ containing an open neighbourhood of $K$ such that $C\subset D$. Then any map $F\colon D \to X$ which satisfies $F_{|C}=\Id_{|C}$ is in the interior of
    \begin{equation*}
        \{f\in \tLip_1(D,X): \dim_H f_\#\mu \geq d\}.
    \end{equation*}
\end{corollary}

\begin{corollary}\label{C:stability-for-Euclidean-Haus}
     Let $\mu$ be a Borel measure on $\reals^k$ with $\dim_T \mu\geq d$. Let $m\geq d$. Then the set
    \begin{equation}\label{E:the-set2}
        \{f\in \tLip_1(\reals^k, \reals^m): \dim_H f_\# \mu\geq d\}
    \end{equation}
    has non-empty interior. For any compact set $K\subset \reals^k$ such that $\dim_T \mu_{|K}\geq d$, any $F\in \tLip_1(\reals^k, \reals^m)$ satisfying $F_{|C}=I_{|C}$ for some isometry $I\colon \reals^k \to \reals^n$ and a convex set $C$ containing an open neighbourhood of $K$, is in the interior of the set in \eqref{E:the-set2}.
\end{corollary}

\begin{remark}\label{R:AC}
    One can use the deep result of DePhilippis and Rindler \cite[Corollary 1.12]{DeRindler} (instead of Lemma \ref{L:DeRindler}) to assert a stronger version of the previous two corollaries. In Corollary \ref{C:stability-for-X-Haus}, one may replace the set in question with the set
    \begin{equation*}
        \{f\in \tLip_1(D,X): \exists \; E\subset K, \mu(E)>0 \;\textnormal{such that $f_\#\mu_{|E}\ll \H^d$}\},
    \end{equation*}
    where the set $E$ is assumed to be Borel. Similar considerations hold also for Corollary \ref{C:stability-for-Euclidean-Haus}.
\end{remark}

Finally, we collect a corollary which provides a characterisation of the stabilized Hausdorff dimension in Euclidean spaces, which was discussed in the introduction. This follows from the combination Corollary \ref{C:main-euclidean-measure} and the previous Corollary \ref{C:stability-for-Euclidean-Haus}.

We remark that one may state this type of a result in many equivalent ways such as replacing ``residual'' with ``dense'' or requiring instead of density only approximation of the identity (cf.~the different definitions \eqref{E:stable-dense} and \eqref{E:stable-id}). Since one may pass from each of these types of statements to any other, we only state the result in one form, the one which we find the most convenient for the reader.

\begin{corollary}\label{C:final}
    Suppose $X$ is a complete and separable metric space and $\mu$ is a Borel measure on $X$. Then
    \begin{equation*}
        \begin{split}
            \dim_T \mu \geq \inf\{s\in[0,\infty]&: 
            \{f\in\tLip_1(X, \reals^m): \dim_H(f_\#\mu)\leq s\}\;\text{is residual in $\tLip_1(X,\reals^m)$}\\
            &\text{for every $m\in\nat$.}\}.
        \end{split}
    \end{equation*}
    If, moreover, $X=\reals^k$ for some $k\in\nat$, then
    \begin{equation*}
        \begin{split}
            \dim_T \mu =\inf\{s\in[0,\infty]&: 
            \{f\in\tLip_1(\reals^k, \reals^m): \dim_H(f_\#\mu)\leq s\}\;\text{is residual in $\tLip_1(\reals^k,\reals^m)$}\\
            &\text{for every $m\in\nat$.}\}.
        \end{split}
    \end{equation*}
    In particular, the quantity on the right hand side is an integer.
\end{corollary}
\begin{proof}
    The first inequality ``$\geq$'' is Corollary \ref{C:main-euclidean-measure}.
    The converse inequality in the case $X=\reals^k$ is
    Corollary \ref{C:stability-for-Euclidean-Haus}.
\end{proof}

\newpage 

\bibliographystyle{abbrvnat} 

\renewcommand{\bibname}{Bibliography}
\bibliography{bibliography}

@article {B,
    AUTHOR = {Bate, David},
     TITLE = {Purely unrectifiable metric spaces and perturbations of
              {L}ipschitz functions},
   JOURNAL = {Acta Math.},
  FJOURNAL = {Acta Mathematica},
    VOLUME = {224},
      YEAR = {2020},
    NUMBER = {1},
     PAGES = {1--65},
      ISSN = {0001-5962},
   MRCLASS = {31E05 (26A16 28A75 49Q15)},
  MRNUMBER = {4086714},
MRREVIEWER = {Jeremy T. Tyson},
       DOI = {10.4310/acta.2020.v224.n1.a1},
       URL = {https://doi.org/10.4310/acta.2020.v224.n1.a1},
}

@book {AT,
    AUTHOR = {Ambrosio, Luigi and Tilli, Paolo},
     TITLE = {Topics on analysis in metric spaces},
    SERIES = {Oxford Lecture Series in Mathematics and its Applications},
    VOLUME = {25},
 PUBLISHER = {Oxford University Press, Oxford},
      YEAR = {2004},
     PAGES = {viii+133},
      ISBN = {0-19-852938-4},
   MRCLASS = {28-01 (28A78 30C65 31C15 46E35 49J45 49K27 53C23)},
  MRNUMBER = {2039660},
MRREVIEWER = {Vasily A. Chernecky},
}

@book {Mat,
    AUTHOR = {Mattila, Pertti},
     TITLE = {Geometry of sets and measures in {E}uclidean spaces},
    SERIES = {Cambridge Studies in Advanced Mathematics},
    VOLUME = {44},
      NOTE = {Fractals and rectifiability},
 PUBLISHER = {Cambridge University Press, Cambridge},
      YEAR = {1995},
     PAGES = {xii+343},
      ISBN = {0-521-46576-1; 0-521-65595-1},
   MRCLASS = {28A75 (49Q20)},
  MRNUMBER = {1333890},
MRREVIEWER = {Harold Parks},
       DOI = {10.1017/CBO9780511623813},
       URL = {https://doi.org/10.1017/CBO9780511623813},
}

@article {DeRindler,
    AUTHOR = {De Philippis, Guido and Rindler, Filip},
     TITLE = {On the structure of {$\mathcal{A}$}-free measures and applications},
   JOURNAL = {Ann. of Math. (2)},
  FJOURNAL = {Annals of Mathematics. Second Series},
    VOLUME = {184},
      YEAR = {2016},
    NUMBER = {3},
     PAGES = {1017--1039},
      ISSN = {0003-486X,1939-8980},
   MRCLASS = {49Q20 (28B20 46A22)},
  MRNUMBER = {3549629},
MRREVIEWER = {Pei\ Biao\ Zhao},
       DOI = {10.4007/annals.2016.184.3.10},
       URL = {https://doi.org/10.4007/annals.2016.184.3.10},
}

@misc{BW,
Author = {David Bate and Julian Weigt},
Title = {Alberti representations, rectifiability of metric spaces and higher integrability of measures satisfying a {PDE}},
Year = {2025},
Eprint = {arXiv:2501.02948},
URL = {https://arxiv.org/abs/2501.02948}
}

@article {BKO,
    AUTHOR = {Bate, David and Kangasniemi, Ilmari and Orponen, Tuomas},
     TITLE = {Cheeger's differentiation theorem via the multilinear {K}akeya
              inequality},
   JOURNAL = {Pure Appl. Funct. Anal.},
  FJOURNAL = {Pure and Applied Functional Analysis},
    VOLUME = {8},
      YEAR = {2023},
    NUMBER = {6},
     PAGES = {1587--1602},
      ISSN = {2189-3756},
   MRCLASS = {28A15 (28A50 30L99 49J52 49Q15)},
  MRNUMBER = {4686569},
MRREVIEWER = {Thomas Z\"{u}rcher},
}

@article {AM,
    AUTHOR = {Alberti, Giovanni and Marchese, Andrea},
     TITLE = {On the differentiability of {L}ipschitz functions with respect
              to measures in the {E}uclidean space},
   JOURNAL = {Geom. Funct. Anal.},
  FJOURNAL = {Geometric and Functional Analysis},
    VOLUME = {26},
      YEAR = {2016},
    NUMBER = {1},
     PAGES = {1--66},
      ISSN = {1016-443X,1420-8970},
   MRCLASS = {49Q15 (26A27 26B05 28A75)},
  MRNUMBER = {3494485},
MRREVIEWER = {Gareth\ Speight},
       DOI = {10.1007/s00039-016-0354-y},
       URL = {https://doi.org/10.1007/s00039-016-0354-y},
}

@article {BateLi,
    AUTHOR = {Bate, David and Li, Sean},
     TITLE = {Characterizations of rectifiable metric measure spaces},
   JOURNAL = {Ann. Sci. \'Ec. Norm. Sup\'er. (4)},
  FJOURNAL = {Annales Scientifiques de l'\'Ecole Normale Sup\'erieure.
              Quatri\`eme S\'erie},
    VOLUME = {50},
      YEAR = {2017},
    NUMBER = {1},
     PAGES = {1--37},
      ISSN = {0012-9593,1873-2151},
   MRCLASS = {28A75 (42B35)},
  MRNUMBER = {3621425},
       DOI = {10.24033/asens.2314},
       URL = {https://doi.org/10.24033/asens.2314},
}

@article {Kaufman,
    AUTHOR = {Kaufman, Robert},
     TITLE = {On {H}ausdorff dimension of projections},
   JOURNAL = {Mathematika},
  FJOURNAL = {Mathematika. A Journal of Pure and Applied Mathematics},
    VOLUME = {15},
      YEAR = {1968},
     PAGES = {153--155},
      ISSN = {0025-5793},
   MRCLASS = {54.70 (26.00)},
  MRNUMBER = {248779},
MRREVIEWER = {F.\ Raymond},
       DOI = {10.1112/S0025579300002503},
       URL = {https://doi.org/10.1112/S0025579300002503},
}

@article {Mars,
    AUTHOR = {Marstrand, J. M.},
     TITLE = {Some fundamental geometrical properties of plane sets of
              fractional dimensions},
   JOURNAL = {Proc. London Math. Soc. (3)},
  FJOURNAL = {Proceedings of the London Mathematical Society. Third Series},
    VOLUME = {4},
      YEAR = {1954},
     PAGES = {257--302},
      ISSN = {0024-6115,1460-244X},
   MRCLASS = {27.2X},
  MRNUMBER = {63439},
MRREVIEWER = {L.\ C.\ Young},
       DOI = {10.1112/plms/s3-4.1.257},
       URL = {https://doi.org/10.1112/plms/s3-4.1.257},
}

@misc{AdolfoPrivate,
Author = {Adolfo Arroyo-Rabasa},
Title = {Private communications},
Year = {2025}
}

@article {SchDer,
    AUTHOR = {Schioppa, Andrea},
     TITLE = {Derivations and {A}lberti representations},
   JOURNAL = {Adv. Math.},
  FJOURNAL = {Advances in Mathematics},
    VOLUME = {293},
      YEAR = {2016},
     PAGES = {436--528},
      ISSN = {0001-8708,1090-2082},
   MRCLASS = {58C20 (46J15 53C23)},
  MRNUMBER = {3474327},
MRREVIEWER = {Thomas\ Z\"urcher},
       DOI = {10.1016/j.aim.2016.02.013},
       URL = {https://doi.org/10.1016/j.aim.2016.02.013},
}

@article {BateLDS,
    AUTHOR = {Bate, David},
     TITLE = {Structure of measures in {L}ipschitz differentiability spaces},
   JOURNAL = {J. Amer. Math. Soc.},
  FJOURNAL = {Journal of the American Mathematical Society},
    VOLUME = {28},
      YEAR = {2015},
    NUMBER = {2},
     PAGES = {421--482},
      ISSN = {0894-0347,1088-6834},
   MRCLASS = {46G05 (30L99 49J52 53C23)},
  MRNUMBER = {3300699},
MRREVIEWER = {Riikka\ Korte},
       DOI = {10.1090/S0894-0347-2014-00810-9},
       URL = {https://doi.org/10.1090/S0894-0347-2014-00810-9},
}

@article {AGPP,
    AUTHOR = {Aliaga, Ram\'on J. and Gartland, Chris and Petitjean, Colin
              and Proch\'azka, Anton\'in},
     TITLE = {Purely 1-unrectifiable metric spaces and locally flat
              {L}ipschitz functions},
   JOURNAL = {Trans. Amer. Math. Soc.},
  FJOURNAL = {Transactions of the American Mathematical Society},
    VOLUME = {375},
      YEAR = {2022},
    NUMBER = {5},
     PAGES = {3529--3567},
      ISSN = {0002-9947,1088-6850},
   MRCLASS = {51F30 (28A78 30L05 46B20 46B22 54E45)},
  MRNUMBER = {4402669},
MRREVIEWER = {Mikhail\ Ostrovskii},
       DOI = {10.1090/tran/8591},
       URL = {https://doi.org/10.1090/tran/8591},
}

@article {ArRaDimension,
    AUTHOR = {Arroyo-Rabasa, Adolfo},
     TITLE = {An elementary approach to the dimension of measures satisfying
              a first-order linear {PDE} constraint},
   JOURNAL = {Proc. Amer. Math. Soc.},
  FJOURNAL = {Proceedings of the American Mathematical Society},
    VOLUME = {148},
      YEAR = {2020},
    NUMBER = {1},
     PAGES = {273--282},
      ISSN = {0002-9939,1088-6826},
   MRCLASS = {28A78 (35F35 49Q15)},
  MRNUMBER = {4042849},
       DOI = {10.1090/proc/14732},
       URL = {https://doi.org/10.1090/proc/14732},
}

@article {BCJ,
    AUTHOR = {Bouchitt\'e, G. and Champion, T. and Jimenez, C.},
     TITLE = {Completion of the space of measures in the {K}antorovich norm},
   JOURNAL = {Riv. Mat. Univ. Parma (7)},
  FJOURNAL = {Rivista di Matematica della Universit\`a{} di Parma. Serie 7},
    VOLUME = {4*},
      YEAR = {2005},
     PAGES = {127--139},
      ISSN = {0035-6298},
   MRCLASS = {49J45 (28A33 46E27)},
  MRNUMBER = {2197484},
MRREVIEWER = {Wilfrid\ Gangbo},
}
\end{document}